\documentclass[12pt]{article}
\usepackage{amsmath,amsfonts,amssymb,amsthm,fancyhdr,bm,mathtools,enumitem}
\usepackage[hidelinks]{hyperref}
\usepackage{fullpage}
\usepackage[dvipsnames]{xcolor}
\usepackage[capitalize]{cleveref}

\newtheorem{thm}{Theorem}[section]
\newtheorem*{thm*}{Theorem}
\newtheorem*{thm:offdiagupper}{\cref{thm:offdiagupper}'}
\newtheorem*{thm:quasirandom}{\cref{thm:quasirandom}'}
\newtheorem*{thm:goodness-tight}{\cref{thm:goodness-tight}}
\newtheorem*{thm:goodness-stability}{\cref{thm:goodness-stability}'}
\newtheorem{lem}[thm]{Lemma}

\newtheorem{conj}[thm]{Conjecture}
\theoremstyle{definition}
\newtheorem{Def}[thm]{Definition}
\newtheorem*{rem}{Remark}

\crefname{lem}{Lemma}{Lemmas}

\newlist{lemenum}{enumerate}{1}
\setlist[lemenum]{label=(\alph*), ref=\thelem(\alph*)}
\crefalias{lemenumi}{lemma}

\DeclareMathOperator\pr{Pr}
\DeclareMathOperator\ext{ext}

\newcommand\up[1]{^{(#1)}}

\newcommand\inv{^{-1}}
\newcommand\flo[1]{\lfloor #1 \rfloor}
\newcommand\ab[1]{\lvert #1\rvert}

\newcommand\Q{\mathcal{Q}}

\newcommand\N{\mathbb{N}}
\newcommand\R{\mathbb{R}}
\newcommand\E{\mathbb{E}}

\newcommand\h{\mathcal{H}}

\newcommand{\pardiff}[2]{\mathchoice{\frac{\partial #1}{\partial #2}}{\partial #1/\partial #2}{\partial #1/\partial #2}{\partial #1/\partial #2}}

\title{Off-diagonal book Ramsey numbers}
\author{David Conlon\thanks{Department of Mathematics, California Institute of Technology, Pasadena, CA 91125, USA. Email: {\tt dconlon@caltech.edu}. Research supported by NSF Award DMS-2054452.} \and Jacob Fox\thanks{Department of Mathematics, Stanford University, Stanford, CA 94305, USA. Email: {\tt jacobfox@stanford.edu}. Research supported by a Packard Fellowship and by NSF Awards DMS-1800053 and DMS-2154169.} \and Yuval Wigderson\thanks{School of Mathematics, Tel Aviv University, Tel Aviv 69978, Israel. Email: {\tt yuvalwig@tauex.tau.ac.il}. Research supported by NSF GRFP Grant DGE-1656518, NSF-BSF Grant 20196, and by ERC Consolidator Grants 863438 and 101044123.}}
\date{}

\begin{document}
\maketitle

\begin{abstract}
    The book graph $B_n \up k$ consists of $n$ copies of $K_{k+1}$ joined along a common $K_k$. In the prequel to this paper, we studied the diagonal Ramsey number $r(B_n \up k, B_n \up k)$. Here we consider the natural off-diagonal variant $r(B_{cn} \up k, B_n\up k)$ for fixed $c \in (0,1]$. In this more general setting, we show that an interesting dichotomy emerges: for very small $c$, a simple $k$-partite construction dictates the Ramsey function and all nearly-extremal colorings are close to being $k$-partite, while, for $c$  bounded away from $0$, random colorings of an appropriate density are asymptotically optimal and all nearly-extremal colorings are quasirandom. Our investigations also open up a range of questions about what happens for intermediate values of $c$.
\end{abstract}

\section{Introduction}

Given two graphs $H_1$ and $H_2$, their \emph{Ramsey number} $r(H_1,H_2)$ is the smallest positive integer $N$ such that every red/blue coloring of the edges of $K_N$ is guaranteed to contain a red copy of $H_1$ or a blue copy of $H_2$. One of the main open problems in Ramsey theory is to determine the asymptotic order of $r(K_n,K_n)$. However, despite intense and longstanding interest, the lower and upper bounds $\sqrt 2^n \leq r(K_n,K_n) \leq 4^n$ for this problem have remained largely unchanged since 1947 and 1935, respectively \cite{Erdos47, ErSz}.

Another major question in graph Ramsey theory, which has seen more progress, is to determine the growth rate of the  \emph{off-diagonal} Ramsey number $r(K_s,K_n)$, where we think of $s$ as fixed and let $n$ tend to infinity. The first non-trivial case is when $s=3$, where it is known that
\[
    r(K_3,K_n) = \Theta \left( \frac{n^2}{\log n}\right),
\]
with the upper bound due to Ajtai, Koml\'os, and Szemer\'edi \cite{AjKoSz} and the lower bound to Kim~\cite{Kim}. Subsequent work of Shearer~\cite{Shearer}, Bohman--Keevash~\cite{BoKe}, and Fiz Pontiveros--Griffiths--Morris \cite{FiGrMo} has led to a better understanding of the implicit constant, which is now known up to a factor of $4+o(1)$. However, the successes in estimating $r(K_3, K_n)$ have not carried over to $r(K_s, K_n)$ for any other fixed $s$ and a polynomial gap persists between the upper and lower bounds for all $s \geq 4$ (though see~\cite{MuVe} for a promising approach to improving the lower bound).

The \emph{book graph} $B_n \up k$ is the graph obtained by gluing $n$ copies of the clique $K_{k+1}$ along a common $K_k$. The ``book'' terminology comes from the case $k=2$, where $B_n \up 2$ consists of $n$ triangles glued along a common edge. Continuing the analogy, each $K_{k+1}$ is called a \emph{page} of the book and the common $K_k$ is called the \emph{spine}. Ramsey numbers of books arise naturally in the study of $r(K_n, K_n)$; indeed, 
Ramsey's original proof~\cite{Ramsey} of the finiteness of $r(K_n,K_n)$ proceeds inductively by establishing the finiteness of certain book Ramsey numbers, while the Erd\H os--Szekeres bound \cite{ErSz} and its improvements~\cite{Conlon2009, Sah2021} are also best interpreted through the language of books.  
Because of this, Ramsey numbers of books have attracted a great deal of attention over the years, starting with papers of Erd\H os, Faudree, Rousseau, and Schelp \cite{ErFaRoSc} and of Thomason \cite{Thomason82}. Both of these papers prove bounds of the form $2^k n-o_k(n) \leq r(B_n\up k, B_n \up k) \leq 4^k n$, where we think of $k$ as fixed and $n \to \infty$, with Thomason conjecturing that the lower bound is closer to the truth. This was confirmed in a recent breakthrough result of the first author~\cite{Conlon}, who proved that, for every fixed $k$,
\[
    r(B_n \up k,B_n \up k)=2^k n+o_k(n).
\]
The original proof of this result relied heavily on an application of Szemer\'edi's celebrated regularity lemma, leading to rather poor control on the error term. In the prequel to this paper~\cite{CoFoWi}, we gave two alternative proofs of this result, one a simplified version of the first author's original proof and the other a proof which avoids the use of the full regularity lemma, 
allowing us to gain significantly better control over the error term (for a discussion of how further improvements might ultimately impinge on the estimation of $r(K_n, K_n)$, we refer the reader to~\cite{CoFoWi}). We also proved a stability result, saying that extremal colorings for this Ramsey problem are quasirandom. 

In this paper, we study a natural off-diagonal generalization of this problem. Specifically, we fix some $k \in \N$ and some $c \in (0,1]$ and we wish to understand the asymptotics of the Ramsey number $r(B_{\flo{cn}} \up k, B_{n} \up k)$ as $n \to \infty$. Note that for $c=1$ this is precisely the question considered above. Henceforth, we omit the floor signs and write $B_{cn}\up k$ instead of $B_{\flo{cn}}\up k$.

Our results reveal that the behavior of the function $r(B_{cn} \up k,B_n \up k)$ varies greatly as $c$ moves from $0$ to $1$. As we shall see, for $c$ sufficiently small, the behavior of this Ramsey number is determined by a simple block construction, while, for $c$ sufficiently far from $0$, its behavior is determined by a random coloring. There is also an intermediate range of $c$ where our results say nothing, but where several interesting questions arise. We will say more about this in the concluding remarks. 

To describe our results in detail, we begin by observing that for any positive integers $k$, $m$, and $n$ with $m \leq n$, we have 
\begin{equation}\label{eq:goodness-bound}
	r(B_m \up k, B_n \up k) \geq k(n+k-1)+1.
\end{equation}
Indeed,	let $N=k(n+k-1)$. We partition the vertices of $K_N$ into $k$ blocks, each of size $n+k-1$. We color all edges inside a block blue and all edges between blocks red. Then any blue $B_n \up k$ must appear inside a block, which it cannot, since $B_n \up k$ has $n+k$ vertices. On the other hand, since the red graph is $k$-partite, it does not contain any red $K_{k+1}$ and so cannot contain a red $B_m \up k$.

This simple inequality is a special case of a more general lower bound, usually attributed to Chv\'atal and Harary \cite{MR314696}, that $r(H_1,H_2) \geq (\chi(H_1)-1)(|V(H_2)|-1)+1$ provided $H_2$ is connected. In general, this lower bound is far from optimal,\footnote{For example, for $H_1=H_2=K_n$, it gives a lower bound of $r(K_n, K_n) =\Omega(n^2)$, whereas the truth is $2^{\Theta(n)}$.} but it is tight for certain sparse graphs. The study of when it is tight goes under the name of \emph{Ramsey goodness}, a term introduced by Burr and Erd\H os \cite{BuEr} in their first systematic investigation of the concept. One of the central results in the field of Ramsey goodness is due to Nikiforov and Rousseau \cite{NiRo}, who proved an extremely general theorem about when this lower bound is tight. As a very special case of their theorem, one has the following result; see also \cite{FoHeWi} for a new proof with better quantitative bounds. 

\begin{thm}[{Nikiforov--Rousseau \cite[Theorem 2.12]{NiRo09}}]\label{thm:goodness-tight}
    For every $k \geq 2$, there exists some $c_0 \in (0,1)$ such that, for any $0<c \leq c_0$ and $n$ sufficiently large,
    \[
        r(B_{cn}\up k, B_n \up k) = k(n+k-1)+1.
    \]
\end{thm}

Moreover, Nikiforov and Rousseau's proof shows that the unique coloring on $k(n+k-1)$ vertices with no red $B_{cn}\up k$ and no blue $B_n \up k$ is the coloring we described, where the red graph is a balanced complete $k$-partite graph (meaning that all the parts have orders as equal as possible). By adapting their proof, we are able to prove a corresponding structural stability result, which says that any coloring on $N=(k+o(1))n$ vertices is either ``close'' to being balanced  complete $k$-partite in red or contains monochromatic books with substantially more pages than what is guaranteed by \cref{thm:goodness-tight}. Note that if $N$ is sufficiently large and  congruent to $1$ modulo $k$, then \cref{thm:goodness-tight} says that any red/blue coloring of $E(K_N)$ contains a red $K_k$ with at least $\frac ck(N-1)-c(k-1)$ extensions to a red $K_{k+1}$ or a blue $K_k$ with at least $\frac 1k(N-1)-(k-1)$ extensions to a blue $K_{k+1}$. 

\begin{thm}\label{thm:goodness-stability}
    For every $k \geq 2$ and every $\theta>0$, there exist $c,\gamma \in (0,1)$ such that the following holds for any sufficiently large $N$ and any red/blue coloring of $E(K_N)$. Either one can recolor at most $\theta N^2$ edges to turn the red graph into a balanced complete $k$-partite graph or else the coloring contains one of the following:
    \begin{itemize}
        \item at least $\gamma N^k$ red $K_k$, each with at least $(\frac ck+\gamma)N$ extensions to a red $K_{k+1}$, or
        \item at least $\gamma N^k$ blue $K_k$, each with at least $(\frac 1k+\gamma)N$ extensions to a blue $K_{k+1}$.
    \end{itemize}
\end{thm}

Informally, this theorem says that either the coloring is close to complete $k$-partite in red or else a constant fraction of the $k$-tuples induce a clique that forms the spine of a monochromatic book with at least $\gamma N$ more pages than what is guaranteed by the Ramsey bound alone.

However, once $c$ is sufficiently far from $0$, the deterministic construction that yields (\ref{eq:goodness-bound}) stops being optimal. Indeed, as in the diagonal case, we can get another lower bound on $r(B_{cn}\up k,B_n \up k)$ by considering random colorings. More precisely, let us fix $k \in \N$ and $c \in (0,1]$ and define
\[
    p=\frac{1}{c^{1/k}+1} \in \left[ \tfrac 12,1\right).
\]
We set $N=(p^{-k}-o(1))n$ and independently color every edge of $K_N$ blue with probability $p$ and red with probability $1-p$. Given a blue $K_k$ in this coloring, the expected number of extensions to a blue $K_{k+1}$ is $p^k(N-k) = n-o(n)$. Similarly, the expected number of extensions of a red $K_k$ to a red $K_{k+1}$ is $(1-p)^k(N-k) = ((1-p)/p)^k n - o(n)=cn-o(n)$, by our choice of $p$. A standard application of the Chernoff bound and the union bound then implies that
w.h.p.\footnote{As usual, we say that an event $E$ happens with high probability (w.h.p.)\ if $\pr(E) \to 1$ as $n \to \infty$, where the implicit parameter $n$ will be clear from context.}\ this coloring contains no blue $B_n \up k$ and no red $B_{cn}\up k$, assuming the $o(n)$ terms are chosen appropriately.
This implies that for any $k \in \N$ and any $c \in (0,1]$,
\[
    r(B_{cn}\up k, B_n \up k) \geq \left(c^{1/k}+1\right)^k n -o_k(n),
\]
while the lower bound in (\ref{eq:goodness-bound}) is that $r(B_{cn} \up k,B_n \up k) \geq (k+o(1))n$. If $c>((1+o(1))\frac{\log k}{k})^k$, then the quantity $(c^{1/k}+1)^k$ is larger than $k+o(1)$, where the logarithm is to base $e$. Thus, once $c$ is sufficiently far from $0$, the bound in (\ref{eq:goodness-bound}) is smaller than the random bound.

Our next main result shows that the random bound actually becomes asymptotically tight 
at this point.

\begin{thm}\label{thm:offdiagupper}
	For every $k \geq 2$, there exists some $c_1=c_1(k) \in (0,1]$ such that, for any fixed $c_1 \leq c \leq 1$,
	\[
		r(B_{cn}\up k,B_n \up k) = \left( c^{1/k}+1 \right) ^kn+o_k(n).
	\]
	Moreover, one may take $c_1(k)=((1+o(1))\frac{\log k}{k})^k$.
\end{thm}

Our third main result is a corresponding structural stability theorem, which says that all  near-extremal Ramsey colorings (i.e., colorings on roughly $(c^{1/k}+1)^kn$ vertices) must either contain a monochromatic book substantially larger than what is guaranteed by \cref{thm:offdiagupper} or be ``random-like''. The latter possibility is captured by the notion of \emph{quasirandomness}, introduced by Chung, Graham, and Wilson \cite{ChGrWi}. For parameters $p,\theta \in (0,1)$, a red/blue coloring of $E(K_N)$ is said to be $(p,\theta)$-quasirandom if, for every pair of disjoint sets $X,Y \subseteq V(K_N)$, we have that
\[
    \left|e_B(X,Y)-p|X||Y|\right| \leq \theta N^2,
\]
where $e_B(X,Y)$ denotes the number of blue edges between $X$ and $Y$. Note that since the colors are complementary, this is equivalent to the analogous condition requiring that $e_R(X,Y)$ is within $\theta N^2$ of $(1-p)|X||Y|$. In their seminal paper, Chung, Graham, and Wilson, building on previous results of Thomason \cite{Thomason82}, showed that this condition is essentially equivalent to a large number of other conditions, all of which encapsulate some intuitive idea of what it means for a coloring to be similar to a random coloring with blue density $p$. With this notion in hand, we can state our structural stability result.

\begin{thm}\label{thm:quasirandom}
	For every $p \in [\frac 12,1)$, there exists some $k_0 \in \N$ such that the following holds for every $k \geq k_0$. For every $\theta>0$, there exists some $\gamma>0$ such that if a red/blue coloring
of $E(K_N)$ is not $(p,\theta)$-quasirandom, then it contains one of the following:
	\begin{itemize}
		\item at least $\gamma N^k$ red $K_k$, each with at least $((1-p)^k+\gamma)N$ extensions to a red $K_{k+1}$, or
		\item at least $\gamma N^k$ blue $K_k$, each with at least $(p^k+\gamma)N$ extensions to a blue $K_{k+1}$.
	\end{itemize}
\end{thm}

\begin{rem}
	As stated, this theorem does not mention the ``off-diagonalness'' parameter $c$ from the previous theorem. But $c$ can easily be recovered as $((1-p)/p)^k$ and the theorem can then be restated to be about blue books with slightly more than $n$ pages or red books with slightly more than $cn$ pages. However, since $p$ is what matters while $c$ plays no real role in the argument, we instead choose to use this language and avoid $c$ entirely.
\end{rem}

In \cref{thm:quasirandom-converse}, we also prove a converse to \cref{thm:quasirandom}, which implies that for $p$ fixed and $k$ sufficiently large in terms of $p$, a coloring of $K_N$ (or, more accurately, a sequence of colorings with $N$ tending to infinity) is $(p, o(1))$-quasirandom if and only if all but $o(N^k)$ red $K_k$ have at most $((1-p)^k+o(1))N$ extensions to a red $K_{k+1}$ and all but $o(N^k)$ blue $K_k$ have at most $(p^k+o(1))N$ extensions to a blue $K_{k+1}$. Thus, we derive a new equivalent formulation for $(p, o(1))$-quasirandomness.

The rest of the paper is organized as follows. In \cref{sec:quoted-results}, we quote (mostly without proof) a number of key results that we will use repeatedly. In \cref{sec:k-partite}, we establish \cref{thm:goodness-stability}, the stability result for small $c$. We prove \cref{thm:offdiagupper}, that the random bound is asymptotically tight once $c$ is not too small, in \cref{sec:ramsey-result} and \cref{thm:quasirandom}, that extremal colorings are quasirandom in this range, in \cref{sec:quasirandomness}. We end with some concluding remarks and open problems.

\subsection{Notation and Terminology}

If $X$ and $Y$ are two vertex subsets of a graph, let $e(X,Y)$ denote the number of pairs in $X \times Y$ that are edges. We will often normalize this and consider the \emph{edge density}
\[
	d(X,Y)=\frac{e(X,Y)}{|X||Y|}.
\]
If we consider a red/blue coloring of the edges of a graph, then $e_B(X,Y)$ and $e_R(X,Y)$ will denote the number of pairs in $X \times Y$ that are blue and red edges, respectively. Similarly, $d_B$ and $d_R$ will denote the blue and red edge densities, respectively. Finally, for a vertex $v$ and a set $Y$, we will sometimes abuse notation and write $d(v,Y)$ for $d(\{v\},Y)$ and similarly for $d_B$ and $d_R$. 

An \emph{equitable partition} of a graph $G$ is a partition of the vertex set $V(G)=V_1 \sqcup \dotsb \sqcup V_m$ with $||V_i|-|V_j||\leq 1$ for all $1 \leq i,j \leq m$. A pair of vertex subsets $(X,Y)$ is said to be \emph{$\varepsilon$-regular} if, for every $X' \subseteq X$, $Y' \subseteq Y$ with $|X'| \geq \varepsilon|X|$, $|Y'| \geq \varepsilon |Y|$, we have
\[
	|d(X,Y)-d(X',Y')| \leq \varepsilon.
\]
Note that we do not require $X$ and $Y$ to be disjoint. In particular, we say that a single vertex subset $X$ is \emph{$\varepsilon$-regular} if the pair $(X,X)$ is $\varepsilon$-regular. We will often need a simple fact, known as the \emph{hereditary property} of regularity, which asserts that for any $0<\alpha\leq1$, if $(X,Y)$ is $\varepsilon$-regular and $X' \subseteq X$, $Y' \subseteq Y$ satisfy $|X'| \geq \alpha |X|$, $|Y'| \geq \alpha |Y|$, then $(X',Y')$ is $(\max\{\varepsilon/\alpha, 2\varepsilon\})$-regular. 

For real numbers $a,b$, we denote by $a \pm b$ any quantity in the interval $[a-b, a+b]$. All logarithms are base $e$ unless otherwise specified. For the sake of clarity of presentation, we systematically omit floor and ceiling signs whenever they are not crucial. In this vein, whenever we have an equitable partition of a vertex set, we will always assume that all of the parts have exactly the same size, rather than being off by at most one. Because the number of vertices in our graphs will always be ``sufficiently large'', this has no effect on our final results.
\label{rem:equitable-partitions}

\section{Results from earlier work} \label{sec:quoted-results}

In this section, we collect some useful tools for the study of book Ramsey numbers, all of which have appeared in previous works. We begin with several results from the theory around Szemer\'edi's regularity lemma and then quote two simple analytic inequalities.

\subsection{Tools from regularity}

We begin with a strengthened form of Szemer\'edi's regularity lemma taken from our first paper \cite[Lemma 2.1]{CoFoWi}.

\begin{lem}\label{reglem}
	For every $\varepsilon>0$ and $M_0 \in \mathbb N$, there is some $M=M(\varepsilon,M_0) \geq M_0$ such that, for every graph $G$ with at least $M_0$ vertices, 
	there is an equitable partition $V(G)=V_1 \sqcup \dotsb \sqcup V_m$ into $M_0 \leq m \leq M$ parts such that the following hold:
	\begin{enumerate}
		\item Each part $V_i$ is $\varepsilon$-regular and
		\item For every $1 \leq i \leq m$, there are at most $\varepsilon m$ values $1 \leq j \leq m$ such that the pair $(V_i,V_j)$ is not $\varepsilon$-regular.
	\end{enumerate}
\end{lem}

To complement the regularity lemma, we will also need a standard counting lemma (see, e.g., \cite[Theorems~2.6.2 and 4.5.1]{Zhao}).

\begin{lem}\label{lem:countinglemma}
	Suppose that $V_1,\ldots,V_k$ are (not necessarily distinct) subsets of a graph $G$ such that all pairs $(V_i,V_j)$ are $\varepsilon$-regular. Then the number of labeled copies of $K_k$ whose $i$-th vertex is in $V_i$ for all $i$ is 
	\[
		\left( \prod_{1 \leq i<j\leq k}d(V_i,V_j) \pm \varepsilon \binom k2 \right) \prod_{i=1}^k |V_i|.
	\]
\end{lem}

We will frequently use the following consequence of the counting lemma, proved in \cite[Corollary 2.6]{CoFoWi}, designed to count monochromatic extensions of cliques and thus estimate the size of monochromatic books. 

\begin{lem}\label{lem:randomclique}
	Fix $k \geq 2$ and let $\eta,\alpha \in (0,1)$ be parameters with $\eta \leq \alpha^3/k^2$. Suppose $U_1,\ldots,U_k$ are (not necessarily distinct) vertex sets in a graph $G$ and suppose that all pairs $(U_i,U_j)$ are $\eta$-regular with $\prod_{1 \leq i<j\leq k}d(U_i,U_j) \geq \alpha$. Let $Q$ be a uniformly random copy of $K_k$ with one vertex in each $U_i$, for $1 \leq i \leq k$, and say that a vertex $u$ extends $Q$ if $u$ is adjacent to every vertex of $Q$. Then, for any $u \in V(G)$,
	\begin{equation*}
		\pr(u \text{ extends }Q) \geq \prod_{i=1}^k d(u,U_i)-4 \alpha.
	\end{equation*}
\end{lem}

The final result in this subsection is actually a simple consequence of Markov's inequality and so does not require any regularity tools to prove. Nonetheless, we will always use it in conjunction with \cref{lem:countinglemma,lem:randomclique}, which is why we include it here. Both the statement and proof are very similar to \cite[Lemma 5.2]{CoFoWi}.

\begin{lem}\label{lem:markov-consequence}
	Let $\kappa,\xi \in (0,1)$, let $0<\nu <\xi$, and suppose that $\Q$ is a set of at least $\kappa N^k$ copies of $K_k$ in an $N$-vertex graph. Suppose that a uniformly random $Q \in\Q$ has at least $\xi N$ extensions to a $K_{k+1}$ in expectation. Then the graph contains at least $(\xi-\nu )\kappa N^k$ copies of $K_k$, each with at least $\nu N$ extensions.
\end{lem}
\begin{proof}
	Let $X$ be the random variable counting the number of extensions of a random $Q \in \Q$ and let $Y=N-X$. Then $Y$ is a non-negative random variable with $\E[Y]=N-\E[X] \leq (1-\xi)N$. By Markov's inequality,
	\[
		\pr(X \leq \nu  N)=\pr \left( Y\geq (1-\nu ) N \right) \leq \frac{\E[Y]}{(1-\nu )N} \leq \frac{(1-\xi)N}{(1-\nu )N}=\frac{1-\xi}{1-\nu }.
	\]
	Thus, 
	\[
		\pr(X \geq \nu  N)\geq 1- \frac{1-\xi}{1-\nu }=\frac{\xi-\nu }{1-\nu } \geq \xi-\nu ,
	\]
	which implies that the number of $Q \in \Q$ with at least $\nu N$ extensions is at least $(\xi-\nu )|\Q| \geq (\xi-\nu )\kappa N^k$, as desired.
\end{proof}

\subsection{Analytic inequalities}

The following lemma is a multiplicative form of Jensen's inequality and is a simple consequence of the standard version. For a proof, see \cite[Lemma A.1]{CoFoWi}.

\begin{lem}[Multiplicative Jensen inequality]\label{multjensen}
	Suppose $0<a<b$ are real numbers and $x_1,\ldots,x_k \in (a,b)$. Let $f: (a,b) \to \R$ be a function such that $y \mapsto f(e^y)$ is strictly convex on the interval $(\log a,\log b)$. Then, for any $z \in (a^k,b^k)$, subject to the constraint $\prod_{i=1}^k x_i=z$, 
	\[
		\frac 1k \sum_{i=1}^k f(x_i)
	\]
	is minimized when all the $x_i$ are equal (and thus equal to $z^{1/k}$).
\end{lem}

The following theorem is the well-known ``defect'' or ``stability'' version of Jensen's inequality. For a proof, see  {\cite[Problem 6.5]{Steele}}.
\begin{thm}[H\"older's Defect Formula]\label{holderdefect}
	Suppose $\varphi:[a,b] \to \R$ is a twice-differentiable function with $\varphi''(y) \geq m > 0$ for all $y \in (a,b)$. For any $y_1,\ldots,y_k \in [a,b]$, let
	\[
		\mu = \frac 1k \sum_{i=1}^k y_i \qquad \text{ and } \qquad \sigma^2 = \frac 1k \sum_{i=1}^k (y_i-\mu)^2
	\]
	be the empirical mean and variance of $\{y_1,\ldots,y_k\}$. Then
	\[
		\frac 1k \sum_{i=1}^k \varphi(y_i) - \varphi (\mu) \geq \frac{m\sigma^2}{2}.
	\]
\end{thm}

\section{The \texorpdfstring{$k$}{k}-partite regime}\label{sec:k-partite}

In this section, we analyze what happens when $c$ is very small. Recall, from the introduction, that a simple $k$-partite construction yields a lower bound for $r(B_{cn}\up k,B_n \up k)$ and, by a result of Nikiforov and Rousseau \cite{NiRo09}, this construction is tight for $c$ sufficiently small. 

\begin{thm:goodness-tight}[{Nikiforov--Rousseau, \cite[Theorem 2.12]{NiRo09}}]
    For every $k \geq 2$, there exists some $c_0 \in (0,1)$ such that, for any $0<c \leq c_0$ and $n$ sufficiently large,
    \[
        r(B_{cn}\up k, B_n \up k) = k(n+k-1)+1.
    \]
\end{thm:goodness-tight}

Our aim here is to adapt the methods of \cite{NiRo09} to prove a stability version of this theorem, our \cref{thm:goodness-stability}. We first make the following definition.

\begin{Def}\label{def:c-gamma-many}
    For $c,\gamma>0$, we say that a red/blue coloring of $E(K_N)$ contains \emph{$(c,\gamma)$-many books} if it contains
    \begin{itemize}
        \item at least $\gamma N^k$ red $K_k$, each with at least $(\frac ck +\gamma)N$ extensions to a red $K_{k+1}$, or
        \item at least $\gamma N^k$ blue $K_k$, each with at least $(\frac 1k +\gamma)N$ extensions to a blue $K_{k+1}$.
    \end{itemize}
\end{Def}

With this definition in place, we may restate \cref{thm:goodness-stability} as follows.

\begin{thm:goodness-stability}
    For every $k \geq 2$ and every $\theta>0$, there exist $c,\gamma \in (0,1)$ such that the following holds. If a red/blue coloring of $E(K_N)$ does not have $(c,\gamma)$-many books, then one can recolor at most $\theta N^2$ edges to turn the red graph into a balanced complete k-partite graph. 
\end{thm:goodness-stability}

As well as referring to  \cref{sec:quoted-results}, we will need 
the following classical result 
of Andr\'asfai, Erd\H os, and S\'os \cite{AnErSo} (see also \cite{Brandt} for a simpler proof).

\begin{thm}[Andr\'asfai--Erd\H os--S\'os \cite{AnErSo}]\label{thm:andrasfai-erdos-sos}
    For every $k \geq 2$, there exists $\rho > 0$ such that if $G$ is a $K_{k+1}$-free graph on $m$ vertices with minimum degree greater than $(1-\frac 1k - \rho)m$, then $G$ is $k$-partite. Moreover, one may take $\rho = 1/(3k^2-k)$.
\end{thm}

This is a stability version of Tur\'an's theorem. Indeed, Tur\'an's theorem implies that if a graph on $m$ vertices has minimum degree at least $(1-\frac 1k)m$, then it contains a copy of $K_{k+1}$, while the Andr\'asfai--Erd\H os--S\'os theorem says that as long as the minimum degree is not too far below $(1-\frac 1k)m$, every such graph must be $k$-partite. 

Before proceeding to the technical details, let us briefly sketch the proof of \cref{thm:goodness-stability}. We are given a red/blue coloring of $E(K_N)$ and we wish to prove that either the coloring is close to complete $k$-partite in red or it contains $(c,\gamma)$-many books for some $c,\gamma>0$. We begin by applying \cref{reglem} to the red graph of the coloring to obtain an equitable partition $V(K_N) = V_1 \sqcup \dotsb \sqcup V_m$, where each part $V_i$ and most pairs $(V_i,V_j)$ are $\eta$-regular for some small $\eta>0$. We now 
wish to improve our understanding of the coloring with respect to this partition.

First, we show that all the parts $V_i$ must have very low internal red density. Indeed, if some part $V_i$ has $d_R(V_i)\geq \delta$, for some fixed $\delta>0$, then the counting lemma, \cref{lem:countinglemma}, implies that $V_i$ contains many red $K_{k+1}$. By a simple averaging argument, this implies that some $k$-tuple of vertices in $V_i$ lies in many red $K_{k+1}$, yielding a red book with $\frac ck N$ pages. In fact, by using \cref{lem:markov-consequence} in place of the averaging argument, we find $(c,\gamma)$-many books if $d_R(V_i) \geq \delta$, so we may assume that $d_R(V_i)<\delta$ for all $i$.

We next build a reduced graph $G$ with vertex set $v_1,\dots,v_m$, where we make $v_i v_j$ an edge if and only if $(V_i,V_j)$ is $\eta$-regular and $d_R(V_i,V_j)\geq \delta$. We claim that every vertex of $G$ has degree at least $(1- \frac 1k - \sigma)m$ for some small $\sigma>0$. Indeed, if some vertex $v_i$ of $G$ has degree smaller than this, then we 
find that $V_i$ has very high blue density to roughly $(\frac 1k + \sigma)m$ of the remaining parts $V_j$. Since $V_i$ also has very high internal blue density, we can use this to find many blue books with spines in $V_i$ and $(\frac 1k + \gamma)N$ pages for some $0<\gamma<\sigma$. This again yields $(c,\gamma)$-many books in the coloring.

So we may assume that the graph $G$ has high minimum degree. By applying \cref{thm:andrasfai-erdos-sos}, we find that either $G$ is $k$-partite or it contains a copy of $K_{k+1}$. In the former case, we can show that the coloring itself is close to $k$-partite in red. In the latter case, this $K_{k+1}$ yields $k+1$ parts, say $V_1,\dots,V_{k+1}$, such that all pairs are $\eta$-regular and have red density at least $\delta$. By another application of the counting lemma and an averaging argument, we can then show that this structure again yields $(c,\gamma)$-many red books.

We now turn to the details of the proof. We will need the following fact about bipartite graphs, which is a simple consequence of a double-counting technique first introduced by K\H ov\'ari, S\'os, and Tur\'an \cite{KoSoTu}.

\begin{lem}\label{lem:zarankiewicz-hypergraph}
    Let $k \geq 2$ and $d \in (0,1)$ and let $\zeta = (d/4)^k$. Let $H$ be a bipartite graph with parts $A,B$, where $\ab B\geq 2k/d$, and suppose that $H$ has at least $d \ab A \ab B$ edges. Let $\h$ be a $k$-uniform hypergraph with vertex set $B$ and at least $(1-\zeta)\binom{\ab B}k$ edges. Then there are at least $\zeta \binom{\ab B}k$ edges of $\h$ such that the vertices of each such edge have at least $\zeta\ab A$ common neighbors in $A$.
\end{lem}

\begin{proof}
    For a $k$-tuple $Q \in \binom Bk$, let $\ext(Q)$ denote the number of common neighbors of $Q$ in $A$. We double-count the number of stars $K_{1,k}$ in $H$ whose central vertex is in $A$ to find that
    \[
        \sum_{Q \in \binom Bk} \ext(Q) = \sum_{a \in A} \binom{\deg(a)}{k}\geq \ab A \binom{\frac 1{\ab A}\sum_{a \in A} \deg(a)}k \geq \ab A \binom{d \ab B}k,
    \]
    where the first inequality follows from convexity. If we split the left-hand side into a sum over tuples $Q$ which are non-edges of $\h$, a sum over tuples $Q$ that are edges of $\h$ with fewer than $\zeta \ab A$ extensions, and the remainder, we find that
    \begin{align*}
        \ab A \binom {d \ab B}k &\leq \sum_{Q \notin E(\h)} \ext(Q) + \sum_{\substack{Q \in E(\h)\\\ext(Q) < \zeta \ab A}} \ext(Q) + \sum_{\substack{Q \in E(\h)\\\ext(Q) \geq \zeta \ab A}} \ext(Q)\\
        &\leq \zeta \ab A  \binom{\ab B}k + \zeta \ab A \binom {\ab B}k + \ab A \ab{\{Q \in E(\h):\ext(Q) \geq \zeta \ab A\}}.
    \end{align*}
    Therefore, the number of edges of $\h$ with at least $\zeta \ab A$ common neighbors is at least $\binom{d \ab B}k - 2\zeta \binom{\ab B}k$. We note that
    \[
        \frac{\binom{d \ab B}k}{\binom{\ab B} k} = \frac{d \ab B}{\ab B} \cdot \frac{d \ab B-1}{\ab B-1}\dotsb \frac{d \ab B - (k-1)}{\ab B-(k-1)} \geq \left(\frac d2\right)^k =2^k \zeta ,
    \]
    where we used our assumption that $\ab B \geq 2k/d$. Thus, the number of edges of $\h$ with at least $\zeta \ab A$ common neighbors in $A$ is at least
    \[
        \binom{d \ab B}k - 2\zeta \binom{\ab B}k \geq (2^k \zeta - 2\zeta) \binom {\ab B}k \geq \zeta \binom{\ab B}{k}.\qedhere
    \]
\end{proof}

With these preliminaries in place, we are now ready to prove \cref{thm:goodness-stability}.

\begin{proof}[Proof of \cref{thm:goodness-stability}]
    Fix some $k \geq 2$, $\theta\in (0,1)$, and a red/blue coloring of $E(K_N)$. Let $\sigma=(\theta/(12k^4))^k$, $M_0=k$, $\delta =\sigma^2$, and $\eta = \delta^{k^2}$. Let $M=M(\eta,M_0)$ be the parameter from \cref{reglem} and let $c=\delta^{k^2}/M^4$ and $\gamma=\delta^{2k^2}/M^{2k}$. We apply \cref{reglem} to the red graph in our coloring with parameters $M_0$ and $\eta$. This yields an equitable partition $V(K_N)=V_1 \sqcup \dotsb \sqcup V_m$ with $M_0 \leq m \leq M$ such that each part $V_i$ is $\eta$-regular in red and, for each $i$, there are at most $\eta m$ values of $j \neq i$ for which $(V_i,V_j)$ is not $\eta$-regular in red.
    
    First suppose that some part, say $V_1$, has internal red density at least $\delta$. 
    By the counting lemma, \cref{lem:countinglemma}, we see that $V_1$ contains at least $\frac{1}{(k+1)!}(\delta^{\binom{k+1}2}- \eta \binom{k+1}2)|V_1|^{k+1}$ red $K_{k+1}$. Since each red $K_{k+1}$ contains exactly $k+1$ red $K_k$, this implies that an average $k$-tuple of vertices in $V_1$ lies in at least
    \[
        \frac{\frac{k+1}{(k+1)!}(\delta^{\binom{k+1}2}- \eta \binom{k+1}2)|V_1|^{k+1}}{\binom{\ab{V_1}}k}\geq \left(\delta^{\binom{k+1}2}-\eta \binom{k+1}2\right) |V_1|\eqqcolon \xi M |V_1|
    \]
    red $K_{k+1}$. That is, if we pick a uniformly random $k$-tuple of vertices from $V_1$, then the expected number of red $K_{k+1}$ containing it is at least $\xi M \ab{V_1}$. If we also define $\kappa = (\delta^{\binom k2}- \eta \binom k2)/(k! M^k)$, then \cref{lem:countinglemma} implies that $V_1$ contains at least $\kappa N^k$ red $K_k$, with an average one having at least $\xi N$ extensions to a red $K_{k+1}$, where we use the fact that $|V_1| \geq N/M$ since the partition is equitable and has $m \leq M$ parts. If we now set $\nu=\xi/2$ and apply \cref{lem:markov-consequence}, we conclude that $V_1$ contains at least $(\xi \kappa/2) N^k$ red $K_k$, each with at least $(\xi/2)N$ extensions to a red $K_{k+1}$. By our choice of parameters,
    \[
        \xi = \frac 1M \left(\delta^{\binom{k+1}2}-\eta \binom{k+1}2\right) \geq \frac{\delta^{k^2}}{2M}
    \]
    and, therefore, $\xi/2 \geq c/k +\gamma$. Similarly, $\kappa \geq \delta^{k^2}/M^k$ and, therefore, $\xi \kappa/2 \geq \gamma$. Thus, we find that in this case the coloring contains $(c,\gamma)$-many books.
    
    Therefore, we may assume that all $V_i$ have $d_R(V_i) <\delta$. We build a reduced graph $G$ with vertex set $v_1,\ldots,v_m$ and  declare $\{v_i, v_j\} \in E(G)$ if $(V_i,V_j)$ is $\eta$-regular and  $d_R(V_i,V_j) \geq \delta$. Suppose that some vertex of $G$, say $v_1$, has degree less than $(1- \frac 1k - \sigma)m$. Since at most $\eta m$ non-neighbors of $v_1$ can come from irregular pairs, we find that $d_B(V_1,V_i) \geq 1- \delta$ for at least $(\frac 1k+\sigma-\eta)m$ choices of $i \in [m]$. Let $I \subseteq [m]$ be the set of such $i$. Since $d_B(V_1) \geq 1-\delta$ and $\delta \leq 1/k^2$, we see that, for $\alpha=k \delta/4$, 
    \[
        d_B(V_1)^{\binom k2} \geq (1- \delta)^{\binom k2} \geq 1 -\binom k2 \delta \geq \alpha.
    \]
    Moreover, we have that $\eta<\alpha^3/k^2$ by our choice of $\eta$. Therefore, we may apply \cref{lem:randomclique}, which implies that if $Q$ is a randomly chosen blue $K_k$ in $V_1$ and $u$ is some vertex in $K_N$, then $\pr(u\text{ extends }Q) \geq d(u,V_1)^k-4 \alpha$. In particular, if we sum this up over all $u \in \bigcup_{i \in I} V_i$, we find that the expected number of blue extensions of $Q$ is at least
    \[
        \sum_{i \in I} \sum_{u \in V_i} (d(u,V_1)^k - 4 \alpha) \geq |I| \frac Nm \left( (1- \delta)^k-4 \alpha\right) \geq \left( \frac 1k+\sigma -\eta\right)\left( (1- \delta)^k-4 \alpha\right) N,
    \]
    where the first inequality follows from the convexity of the function $x \mapsto x^k$. Using $\eta<\sigma/2$, we have that
    \begin{align*}
        \left( \frac 1k+\sigma -\eta\right)\left( (1- \delta)^k-4 \alpha\right) &\geq \left( \frac 1k +\frac \sigma2\right)(1-2k \delta) \geq \frac 1k +\frac \sigma 2 - 2k \delta,
    \end{align*}
    where the last step follows from the bound $1/k+\sigma \leq 1/k+1/k \leq 1$. By our choice of $\delta=\sigma^2 \leq \sigma/(8k)$, we see that the expected number of blue extensions of $Q$ is at least $(\frac 1k+\frac\sigma4)N$. Moreover, by \cref{lem:countinglemma}, the number of choices for $Q$ is at least $\frac{1}{k!} ((1-\delta)^{\binom k2}-\eta \binom k2)(N/M)^k \geq \kappa N^k$. Therefore, if we apply \cref{lem:markov-consequence} with parameters $\kappa$, $\xi=\frac 1k+\frac \sigma4$, and $\nu=\frac 1k+\gamma$, then we find that the coloring contains $(c,\gamma)$-many books.
    
    Therefore, we may assume that every vertex in $G$ has degree greater than $(1- \frac 1k -\sigma)m$, so, by \cref{thm:andrasfai-erdos-sos} and the fact that $\sigma <1/(3k^2-k)$, we see that either $G$ contains a $K_{k+1}$ or $G$ is $k$-partite. Assume first that there is a $K_{k+1}$ in $G$. By relabeling the vertices, we may assume that $v_1,\ldots,v_{k+1}$ form a clique. 
    By the counting lemma, \cref{lem:countinglemma}, we have that $V_1,\ldots,V_k$ span at least $(\delta^{\binom k2}- \eta \binom k2)(N/m)^k \geq \kappa N^k$ red $K_k$ and $V_1,\ldots,V_{k+1}$ span at least $(\delta^{\binom{k+1}2} - \eta \binom{k+1}2)(N/m)^{k+1}$ red $K_{k+1}$. Every such red $K_{k+1}$ contains exactly one red $K_k$ with one vertex in each of $V_1,\dots,V_k$, so an average $k$-tuple $(v_1,\dots,v_k) \in V_1 \times \dotsb \times V_k$ lies in at least
    \[
        \frac{(\delta^{\binom{k+1}2} - \eta \binom{k+1}2)(N/m)^{k+1}}{(N/m)^k} \geq \left(\delta^{\binom{k+1}2} - \eta \binom{k+1}2\right) \frac NM = \xi N.
    \]
    Thus, we have a set of at least $\kappa N^k$ red $K_k$ with at least $\xi N$ extensions on average and so, applying \cref{lem:markov-consequence} as before, our coloring has $(c,\gamma)$-many books.
    
    Thus, we may assume that $G$ is $k$-partite. Let this $k$-partition of $V(G)$ be $A_1\sqcup \dotsb \sqcup A_k$. Note that $\ab {A_\ell} \leq (\frac 1k+\sigma)m$ for every $\ell$, since the minimum degree of $G$ is at least $(1-\frac 1k - \sigma)m$ and each $A_\ell$ is an independent set in $G$. This in turn implies that $\ab {A_\ell} \geq (\frac 1k-k\sigma)m$ for every $\ell$, since $\ab{A_\ell} = m -\sum_{\ell' \neq \ell}\ab{A_{\ell'}} \geq (\frac 1k-k\sigma)m$. We lift this partition to a partition of the vertices of $K_N$ into $k$ parts $X_1,\dots,X_k$ by letting $X_\ell = \bigcup_{v_i \in A_\ell}V_i$, noting that our observations above imply that $\ab{X_\ell} = (\frac 1k \pm k\sigma)N$ for all $\ell$. We claim that each $X_\ell$ contains at most $\frac{3\delta}2 \binom{\ab{X_\ell}}2$ red edges. Indeed, observe that if $v_i,v_j$ are two (not necessarily distinct) vertices of $G$ that are in the same part $A_\ell$, then they must be non-adjacent in $G$. This means that either $(V_i,V_j)$ is an irregular pair or  $d_R(V_i,V_j)<\delta$. There are at most $\eta m^2$ irregular pairs, so the irregular pairs can contribute at most $\eta N^2\leq 4k^2 \eta \ab{X_\ell}^2 \leq 10 k^2 \eta \binom{\ab{X_\ell}}2$ red edges inside $X_\ell$, where we used that $\ab{X_\ell} \geq (\frac 1k-k\sigma)N \geq N/(2k)$. All other pairs of parts inside each $X_\ell$ have red density at most $\delta$, so the total number of red edges inside $X_\ell$ is at most $\delta \binom{\ab{X_\ell}}2 + 10k^2 \eta \binom{\ab{X_\ell}}2 \leq \frac{3\delta}2 \binom{\ab{X_\ell}}2$, since $\eta \leq\delta/(20k^2)$. This implies that the number of ordered pairs of (not necessarily distinct) vertices in $X_\ell$ which do not form a blue edge is at most $2 \delta \ab{X_\ell}^2$.
    
    This already implies that the red graph can be made $k$-partite by recoloring at most $2\delta N^2$ edges, so it only remains to show that by recoloring a small number of additional edges, we can make the red graph balanced complete $k$-partite. For this, suppose that $d_B(X_1,X_2) \geq \theta/k^2$. If we sample (with repetition) a random $k$-tuple $Q$ of vertices from $X_2$, then the probability that it does not form a blue clique is at most $\binom k2\cdot 2 \delta \leq k^2 \delta$, since each random pair of vertices does not span a blue edge with probability at most $2\delta$. Moreover, the expected number of common blue neighbors of $Q$ inside $X_2$ is at least $(1-2\delta)^k \ab{X_2} - k\geq (1-2k\delta)\ab{X_2}$, by convexity. 
    By applying Markov's inequality as in the proof of \cref{lem:markov-consequence}, the probability that $Q$ has fewer than $(1-\sqrt \delta)\ab{X_2}$ common blue neighbors in $X_2$ is at most $2k\sqrt \delta$. 
    Therefore, the probability that $Q$ is a blue clique with at least $(1-\sqrt \delta)\ab{X_2}$ common blue neighbors in $X_2$ is at least $1-k^2 \delta -2k\sqrt \delta \geq 1-3k\sqrt \delta$, since $\sqrt \delta \leq 1/k$. Let $\h$ be the $k$-uniform hypergraph with vertex set $X_2$ whose edges are all blue $K_k$ in $X_2$ with at least $(1-\sqrt \delta)\ab{X_2}$ common blue neighbors in $X_2$. Then $\h$ has at least $(1-3k\sqrt \delta)\binom{\ab{X_2}}k$ edges.
    
    We now apply \cref{lem:zarankiewicz-hypergraph} to the hypergraph $\h$ and to the bipartite graph of blue edges between $X_1$ and $X_2$, which has edge density $d\geq\theta/k^2$ by assumption. We have that $(d/4)^k \geq(\theta/(4k^2))^k \geq 3k\sigma = 3k\sqrt \delta$ by our choice of $\sigma$ and $\delta$, so we may indeed apply \cref{lem:zarankiewicz-hypergraph} to conclude that at least $(\theta/(4k^2))^k\binom{\ab{X_2}}k$ of the edges of $\h$ have at least $(\theta/(4k^2))^k\ab{X_1}$ common blue neighbors in $X_1$. This yields at least
    \[
        \left(\frac \theta{4k^2}\right)^k \binom{\ab{X_2}}k \geq \left(\frac{\theta\ab{X_2}}{4k^3}\right)^k \geq \left(\frac{\theta}{8k^4}\right)^k N^k \geq \gamma N^k
    \]
    blue $K_k$, each of which has at least
    \begin{align*}
        (1-\sqrt \delta)\ab {X_2} + \left(\frac \theta{4k^2}\right)^k \ab{X_1} &\geq \left(1-\sqrt \delta + \left(\frac \theta{4k^2}\right)^k\right) \left(\frac 1k -k\sigma\right)N\\
        &\geq \left(\frac 1k + \left(\frac \theta{4k^3}\right)^k -\sqrt \delta - 2k\sigma\right)N\\
        &\geq \left(\frac 1k + \left(\frac \theta{4k^3}\right)^k - 3k\sigma\right)N\\
        &\geq \left(\frac 1k+\gamma\right)N
    \end{align*}
    extensions to a blue $K_{k+1}$, where in both computations we used the fact that $\ab{X_1},\ab{X_2} \geq (\frac 1k - k\sigma)N\geq N/(2k)$, as well as our choices of $\sqrt \delta = \sigma = (\theta/(12k^4))^k$. Thus, in this case, we have again found $(c,\gamma)$-many books, a contradiction.
    
    Hence, we may assume that $d_B(X_1,X_2) <\theta/k^2$. By the same argument, all the blue densities between different parts $X_\ell$ can be assumed to be at most $\theta/k^2$. Since we have already argued that the red density inside each part is at most $2\delta$, we see that, by recoloring at most $(\binom k2\theta/k^2 + 2k\delta )N^2$ edges, we can make the red graph complete $k$-partite. Finally, we recall that each part $X_\ell$ has size $\ab{X_\ell} = (\frac 1k \pm k\sigma)N$. Therefore, by moving at most $k^2\sigma N$ arbitrary vertices into another part, we see that we can make our partition  equitable. We then recolor the edges incident with any moved vertex to obtain a balanced complete $k$-partite red graph. Doing so entails recoloring at most $k^2 \sigma N^2$ additional edges. Thus, in total, we recolor at most
    \[
        \left(\binom k2 \frac \theta{k^2} + 2k\delta + k^2 \sigma\right)N^2 \leq \left(\frac \theta 2 + 3k^2 \sigma\right)
        \leq \theta N^2
    \]
    edges, where we used that $\delta \leq \sigma$ and $\sigma \leq (\theta/(12k^4))^k \leq \theta/(6k^2)$.\qedhere
\end{proof}

\section{An upper bound matching the random bound} \label{sec:ramsey-result}

In this section, we prove \cref{thm:offdiagupper}, which says that when $c$ is not too small, the random lower bound for $r(B_{cn}\up k, B_n \up k)$ is asymptotically tight. 
To prove this theorem, we will mimic our simplified proof of the diagonal result from \cite[Section 3]{CoFoWi}, though it needs to be adapted to the off-diagonal setting in several ways. Before proceeding with the details, we sketch the proof at a high level, indicating which parts require new ideas beyond those already present in \cite{CoFoWi}.

A key notion used in the proof is that of a red-blocked configuration. Informally, a red-blocked configuration consists of $k$ disjoint vertex sets such that each set and all pairs are $\eta$-regular for some small $\eta$, every set has red density at least $\delta$ for some small $\delta$, and every pair has blue density at least $\delta$. A blue-blocked configuration is defined similarly, except with the roles of red and blue interchanged. Like the good and great configurations defined in \cite{CoFoWi}, we care about such configurations because their existence automatically implies the existence of large monochromatic books. The precise statement is given in \cref{lem:block-config-suffices}, but, roughly, it says that if we have a red/blue coloring of the complete graph on $(c^{1/k}+1)^k n +o(n)$ vertices which contains a red-blocked or a blue-blocked configuration and  $k$ is sufficiently large with respect to $c$, then the coloring contains a red $B_{cn}\up k$ or a blue $B_n\up k$. This is the key lemma which underlies the entire proof. Its proof is similar to that of \cite[Lemma~3.3]{CoFoWi}, but requires a few modifications. First, the analytic inequality which yields the result is more complicated in the off-diagonal setting and this is where the (necessary) assumption that $k$ is large with respect to $c$ comes from. Second, the averaging arguments used in the proof of \cref{lem:block-config-suffices} require a little more care than those used in the proof of \cite[Lemma~3.3]{CoFoWi}, because we must take $(p,1-p)$-weighted averages here. Finally, though in principle one needs separate arguments to deal with red-blocked configurations and blue-blocked configurations, it turns out that the same proof works for both cases, simply by interchanging the roles of red and blue and of $p$ and $1-p$. 

The remainder of the proof now comes down to finding a red-blocked or blue-blocked configuration or else finding a large monochromatic book directly. To do this, we begin by applying \cref{reglem} to the red graph of the coloring, obtaining a regular equitable partition $V(K_N) = V_1\sqcup \dotsb \sqcup V_m$. Call a part \emph{red} if it contains more red edges than blue edges and \emph{blue} otherwise. We assume for now that at least $pm$ of the parts are blue; the case where at least $(1-p)m$ of the parts are red runs similarly. We build a reduced graph $G$ whose vertices are in bijection with the blue parts and where edges represent pairs of parts that are regular and have red density at least $\delta$ for some small $\delta>0$. By defining $G$ in this way, we see that a $K_k$ in $G$ corresponds to a blue-blocked configuration in the original coloring, so it suffices to find a $K_k$ in $G$.

To do this, we first show that every vertex in $G$ must have degree at least roughly  $(1-p^{k-1})\ab{V(G)}$. Indeed, if this is not the case, since $\ab{V(G)} \geq pm$, we find that there is some blue part $V_i$ which has very high blue density to at least roughly $p^k m$ other parts. 
This can then be used to find a blue $B_n\up k$, where $n \approx p^kN$. So we may conclude that every vertex in $G$ has high degree. But, by Tur\'an's theorem, plus the fact that $p^{k-1}<1/(k-1)$ for sufficiently large $k$, this implies that $G$ contains a copy of $K_k$, as desired.

We now begin the detailed proof of \cref{thm:offdiagupper}. The following result generalizes a key analytic inequality from the diagonal case~\cite[Lemma 3.4]{CoFoWi}.

\begin{lem}\label{lem:offdiagxi}
	For every $p \in (0,1)$, there exists some $k_1 \in \N$ such that if $k \geq k_1$ and $x_1,\ldots,x_k \in [0,1]$, then
	\[
		p^{1-k} \prod_{i=1}^k x_i+\frac{(1-p)^{1-k}}{k} \sum_{i=1}^k (1-x_i)^k \geq 1.
	\]
	Moreover, one may take
	\[
	    k_1(p) = 
	    \begin{cases}
	   6 &\text{if }p \geq 1-5/(4e)\\
	    1+\frac{5 - \log \log \frac 1{1-p} + \log(-\log \log \frac 1{1-p})}{\log \frac 1{1-p}}&\text{otherwise.}
	    \end{cases}
	\]
\end{lem}

\begin{proof}
	First suppose that $x_j \leq\frac 1k$ for some $j \in [k]$. Then we have that
	\[
		\frac{(1-p)^{1-k}}{k} \sum_{i=1}^k (1-x_i)^k \geq \frac{(1-p)^{1-k}}{k}(1-x_j)^k \geq \frac{(1-p)^{1-k}}{k} \left( 1- \frac 1k \right) ^k \geq \frac{(1-p)^{1-k}}{e^2k}\eqqcolon f(p,k),
	\]
	where we used the inequality $1-x \geq e^{-2x}$ for $x \in [0,\frac 12]$. 
	If $p \geq 1-5/(4e)$, then $1-p \leq 5/(4e)$, so $f(p,k) \geq (4/5)^{k-1} e^{k-3}/k$. Once $k \geq 6$, this last expression is at least $1$, so in the case where $p \geq 1-5/(4e)$, we may take $k_1(p)=6$.
	
	For $p<1-5/(4e)$, let $\lambda=\lambda(p) = \log \frac 1{1-p}$ and
	\[
	    k_1(p)=1+\frac{5 - \log \lambda + \log\log \frac 1\lambda}{\lambda}= 1+\frac{5 - \log \log \frac 1{1-p} + \log(-\log \log \frac 1{1-p})}{\log \frac 1{1-p}}.
	\]
	We now claim that 
	\begin{equation}\label{eq:fpk}
	    f(p,k) \geq 1\qquad\text{ if }k \geq k_1(p).
	\end{equation}
	By differentiating, we see that $f(p,k)$ is monotonically increasing in $k$ for $k \geq 1/\log \frac 1{1-p} = 1/\lambda$. Since $p<1-5/(4e)$, we have that $5-\log \lambda +\log \log \frac 1\lambda>1$ and so we are in the monotonicity regime. It therefore suffices to prove the statement for $k=k_1(p)$. Note now that
	\[
	    (1-p)^{1-k_1(p)} = (1-p)^{(5-\log \lambda + \log \log \frac 1 \lambda)/\log(1-p)} = \frac{e^5 \log \frac 1 \lambda}{\lambda}
	\]
    and let $g(p)=1 + \lambda/\log \frac1 \lambda + 5/\log \frac 1\lambda + \log \log \frac 1\lambda/\log \frac 1\lambda$.
	Then we have
	\[
	    f(p,k_1(p)) = \frac{e^5 \log \frac 1 \lambda}{\lambda} \cdot \frac 1{e^2 k_1(p)} =\frac{ e^3 \log \frac 1 \lambda }{\lambda+5+\log \frac 1 \lambda +\log \log \frac 1 \lambda} = \frac{e^3}{g(p)}.
	\]
	Thus, to prove that $f(p,k_1(p))\geq 1$, it suffices to prove that $g(p) \leq e^3$ for all $p <1-5/(4e)$. By differentiating, one can check that $g(p)$ is monotonically increasing in $p \in [0,1-5/(4e)]$. Thus, it suffices to check that $g(1-5/(4e)) \leq e^3$. But $g(1-5/(4e))\approx 18.4 < e^3$, so  $f(p,k_1(p))\geq 1$, as claimed. 
	Hence, from now on, we may assume that all the $x_i$ are in $(\frac 1k,1]$.

	For the moment, let's assume that all the $x_i$ are in $(\frac 1k,1)$. 
	Note that the function $\varphi:y \mapsto (1-e^y)^k$ is strictly convex on the interval $(\log \frac 1k,0)$. By the multiplicative Jensen inequality, \cref{multjensen}, this implies that, subject to the constraint $\prod_{i=1}^k x_i=z$, the term $\frac 1k \sum_{i=1}^k (1-x_i)^k$ is minimized when all the $x_i$ are equal to $z^{1/k}$. Therefore,
	\[
		p^{1-k} \prod_{i=1}^k x_i+\frac{(1-p)^{1-k}}{k} \sum_{i=1}^k (1-x_i)^k \geq p^{1-k}z+(1-p)^{1-k}(1-z^{1/k})^k.
	\]
	So it suffices to minimize this expression as a function of $z$. Changing variables to $w=z^{1/k}$, it suffices to minimize
	\[
		\psi(w)=p^{1-k} w^k+(1-p)^{1-k} (1-w)^k
	\]
	as a function of $w$. By differentiating, we find that $\psi$ is minimized at $w=p$, where $\psi(p)=1$. 
	This proves the desired result as long as all the $x_i$ are in $[0,1)$. By continuity, the result then extends to all $x_i \in [0,1]$. 
\end{proof}

\begin{Def}
	Fix parameters $k \in \N$ and $\eta,\delta \in (0,1)$ and suppose that we are given a red/blue coloring of $E(K_N)$. Then a $k$-tuple of pairwise disjoint vertex sets $C_1,\ldots,C_k \subseteq V(K_N)$ is called a \emph{$(k,\eta,\delta)$-red-blocked configuration} if the following properties are satisfied:
	\begin{enumerate}
		\item Each $C_i$ is $\eta$-regular with itself,
		\item Each $C_i$ has internal red density at least $\delta$, and
		\item For all $i \neq j$, the pair $(C_i,C_j)$ is $\eta$-regular and has blue density at least $\delta$.
	\end{enumerate}
	Similarly, we say that $C_1,\ldots,C_k$ is a \emph{$(k,\eta,\delta)$-blue-blocked configuration} if properties (1--3) hold, but with the roles of red and blue interchanged. 
\end{Def}

The reason we care about these configurations is that, for appropriate choices of the parameters $\eta$ and $\delta$, their existence yields the existence of the required monochromatic books.
This idea (or, rather, the version of it when red and blue play symmetric roles) already appears implicitly in the work of the first author~\cite{Conlon}, but was made much more explicit in the prequel to this paper~\cite{CoFoWi}. 
The precise statement we will need here is given by the next lemma.

\begin{lem}\label{lem:block-config-suffices}
	For every $p \in [\frac 12, 1)$, there is $k_2 \in \N$ such that the following holds. Let $k \geq k_2$, $c = ((1-p)/p)^k$, and $0<\varepsilon<\frac 12$ and suppose $0<\delta \leq (1-p) \varepsilon$ and $0<\eta\leq \delta^{4k^2}$. Suppose that the edges of $K_N$ with $N=(p^{-k}+\varepsilon)n$ are red/blue colored and this coloring contains either a $(k,\eta,\delta)$-red-blocked configuration or a $(k,\eta,\delta)$-blue-blocked configuration. Then, in either case, the coloring contains either a red $B_{cn}\up k$ or a blue $B_n \up k$.
	Moreover, one may take $k_2(p) =  k_1(1-p)$, where $k_1$ is the constant from \cref{lem:offdiagxi}.
\end{lem}
\begin{proof}
	We tackle the two cases separately: suppose first that the coloring has a $(k,\eta,\delta)$-red-blocked configuration, say $C_1,\ldots,C_k$. By the counting lemma, \cref{lem:countinglemma}, we know that the number of blue $K_k$ with one vertex in each $C_i$ is at least
	\begin{align*}
		\left( \prod_{1 \leq i<j \leq k} d_B(C_i,C_j)-\eta \binom k2 \right) \prod_{i=1}^k |C_i|&\geq \left( \delta^{\binom k2}- \eta \binom k2 \right) \prod_{i=1}^k |C_i| >0,
	\end{align*}
	so there is at least one blue $K_k$ with one vertex in each of $C_1,\ldots,C_k$. By a similar computation, we see that each $C_i$ contains at least one red $K_k$.

	For a vertex $v$ and $i \in [k]$, let $x_i(v)\coloneqq d_B(v,C_i) \in [0,1]$. We observe that from the definition in \cref{lem:offdiagxi}, we have that $k_1(1-p) \geq k_1(p)$ for all $p \geq \frac 12$. Therefore, \cref{lem:offdiagxi} implies that since $k \geq k_2 \geq k_1(p)$, we have that
	\[
		p \left( p^{-k}\prod_{i=1}^k x_i(v) \right) +(1-p) \left( \frac{(1-p)^{-k}}{k}\sum_{i=1}^k (1-x_i(v))^k \right) \geq 1
	\]
	for all $v \in V$. Summing this fact up over all $v$, we find that
	\begin{equation}
		p \left( p^{-k}\sum_{v \in V}\prod_{i=1}^k x_i(v) \right) +(1-p) \left( \frac{(1-p)^{-k}}{k}\sum_{i=1}^k \sum_{v \in V} (1-x_i(v))^k \right) \geq N.\label{eq:weighted-avg}
	\end{equation}
	This says that a $(p,1-p)$-weighted average of two numbers is at least $N$, which means that at least one of them is at least $N$. Suppose first that the first term is at least $N$, i.e., that
	\[
		\sum_{v \in V} \prod_{i=1}^k x_i(v) \geq p^k N.
	\]
	Let $Q$ be a uniformly random blue $K_k$ spanning $C_1,\ldots,C_k$, which must exist by our computations above. Let $\alpha =\delta^{k^2} \leq \prod_{i<j} d_B(C_i,C_j)$ and observe that $\eta \leq \delta^{4k^2} = \alpha^4 \leq \alpha^3/k^2$. Thus, for any $v$, we can apply \cref{lem:randomclique} to conclude that the probability $v$ extends $Q$ to a blue $K_{k+1}$ is at least $\prod_i x_i(v)-4 \alpha$. Therefore, the expected number of extensions of $Q$ to a blue $K_{k+1}$ is at least
	\begin{align}
		\sum_{v \in V} \left( \prod_{i=1}^k x_i(v)-4 \alpha \right) &\geq (p^k-4 \alpha)N\notag\\
		&\geq (p^k-4 \alpha)(p^{-k}+\varepsilon)n\notag\\
		&\geq(1+p^k \varepsilon-8 \alpha p^{-k}) n\notag\\
		&\geq n,\label{eq:blue-computation}
	\end{align}
	where (\ref{eq:blue-computation}) uses that $\alpha=\delta^{k^2} \leq((1-p)\varepsilon)^{k^2}\leq (p \varepsilon)^{k^2} \leq p^{2k} \varepsilon/8$. 
	Therefore, $Q$ has at least $n$ extensions in expectation, so there must exist some blue $K_k$ with at least $n$ extensions, i.e., a blue $B_n \up k$.

	Now assume that the other term in the weighted average in (\ref{eq:weighted-avg}) is at least $N$, i.e., that
	\[
		\frac 1k \sum_{i=1}^k \sum_{v \in V} (1-x_i(v))^k \geq (1-p)^{k}N.
	\]
	Then there must exist some $i$ for which
	\[
		\sum_{v \in V} (1-x_i(v))^k \geq (1-p)^{k} N.
	\]
	Therefore, if $Q$ is a random red $K_k$ inside this $C_i$, then, by \cref{lem:randomclique}, the expected number of extensions of $Q$ is at least\footnote{Strictly speaking, if $v \in C_i$, then $d_R(v,C_i) \neq 1-x_i(v)$,
	as $v$ has no edge to itself. However, this tiny loss can be absorbed into the error terms 
	and the result does not change.} 
	\begin{align}
		\sum_{v \in V} \left[ (1-x_i(v))^k-4 \alpha \right] &\geq \left[(1-p)^k-4 \alpha\right](p^{-k}+\varepsilon)n\notag \\
		&\geq \left(\left( \frac{1-p}{p} \right) ^k+(1-p)^k \varepsilon -8 \alpha p^{-k}\right)n\notag \\
		&\geq cn,\label{eq:red-computation}
	\end{align}
	where we use the fact that $c=((1-p)/p)^k$ and that
	\[
		\alpha=\delta^{k^2}=((1-p) \varepsilon)^{k^2} \leq p^k (1-p)^k \varepsilon/8,
	\]
	since $1-p \leq p$. 
	Thus, the expected number of red extensions of a red $K_k$ in $C_i$ is at least $cn$, so there must exist a red $B_{cn}\up k$. This concludes the proof under the assumption that the coloring contains a $(k,\eta,\delta)$-red-blocked configuration.

	Now, we instead assume that the coloring contains a $(k,\eta,\delta)$-blue-blocked configuration and aim to conclude the same result; the proof is more or less identical, but with the role of $p$ now played by $q=1-p$. As before, we find that there is at least one red $K_k$ spanning $C_1,\ldots,C_k$ and that each $C_i$ contains at least one blue $K_k$. For a vertex $v$ and $i \in [k]$, let $y_i(v)=d_R(v,C_i) \in [0,1]$ and write $q=1-p$. Since $k \geq k_2 = k_1(q)$, we can sum the result of applying \cref{lem:offdiagxi} over all $v \in V$ to find that
	\[
		q \left( q^{-k} \sum_{v \in V} \prod_{i=1}^k y_i(v) \right) +(1-q) \left( \frac{(1-q)^{-k}}{k}\sum_{i=1}^k \sum_{v \in V} (1-y_i(v))^k \right) \geq N.
	\]
	As before, this is a $(q,1-q)$-weighted average of two terms, which means that one of the terms must be at least $N$. Suppose first that the first term is at least $N$, i.e., that
	\[
		\sum_{v \in V} \prod_{i=1}^k y_i(v) \geq q^k N.
	\]
    If $Q$ is a uniformly random red $K_k$ spanning $C_1,\ldots,C_k$ and  $\alpha=\delta^{k^2}$, then, as before, we find that the expected number of extensions of $Q$ to a red $K_{k+1}$ is at least
	\[
		\sum_{v \in V} \left( \prod_{i=1}^k y_i(v)-4 \alpha \right) \geq (q^k-4 \alpha) N \geq ((1-p)^k- 4 \alpha)(p^{-k}+\varepsilon)n
		\geq cn,
	\]
	by the computation in (\ref{eq:red-computation}). Therefore, in this case, there must exist some red $K_k$ with at least $cn$ red extensions, giving the desired red $B_{cn}\up k$. So we may assume instead that
	\[
		\frac 1k \sum_{i=1}^k \sum_{v \in V} (1-y_i(v))^k \geq (1-q)^kN,
	\]
	which implies that for some $i \in [k]$,
	\[
		\sum_{v \in V} (1-y_i(v))^k \geq p^k N.
	\]
	Thus, if $Q$ is a random blue $K_k$ inside this $C_i$, we find that the expected number of blue extensions of $Q$ is at least
	\begin{align*}
		\sum_{v \in V} \left[ (1-y_i(v))^k-4\alpha \right] \geq (p^k-4 \alpha)N\geq (p^k-4 \alpha)(p^{-k}+\varepsilon)n \geq n,
	\end{align*}
	by the same computation as in (\ref{eq:blue-computation}). This gives us our blue $B_n \up k$, completing the proof. 
\end{proof}

With this result in hand, we can now prove \cref{thm:offdiagupper}.
\begin{proof}[Proof of \cref{thm:offdiagupper}]
    Given an integer $k \geq 2$, let $c_1(k)$ be the infimum of $c \in (0,1]$ such that $k_2((c^{1/k}+1)\inv) \leq k$, where $k_2$ is the constant from \cref{lem:block-config-suffices}. Note that we declare this infimum to equal $1$ if no $c \in (0,1]$ satisfies this condition (as happens for $k=2$). In this case, there is nothing to prove, since \cref{thm:offdiagupper} for $c=1$ is already known~\cite{Conlon}. We now fix $c\in [c_1,1]$ and $p = 1/(c^{1/k}+1) \in (\frac 12,1]$, noting that we have $k \geq k_2(p)$.
	
	Fix $0<\varepsilon<\frac 12$ and suppose we are given a red/blue coloring of $E(K_N)$ where $N=(p^{-k}+\varepsilon)n$. Our goal is to prove that if $n$ is sufficiently large in terms of $\varepsilon$, then this coloring contains a red $B_{cn} \up k$ or a blue $B_n \up k$. To do this, we fix parameters $\delta=(1-p)^{2k}\varepsilon/(4k)$ and $\eta=\min\{\delta^{4k^2},(1-p)/(4k)\}$ depending on $c$, $k$, and $\varepsilon$. 

	We apply \cref{reglem} to the red graph from our coloring with parameters $\eta$ and $M_0=1/\eta$ to obtain an equitable partition $V(K_N)=V_1 \sqcup \dotsb \sqcup V_m$, where each $V_i$ is $\eta$-regular and, for each $i$, there are at most $\eta m$ values $1 \leq j \leq m$ such that the pair $(V_i,V_j)$ is not $\eta$-regular. Moreover, $M_0 \leq m \leq M=M(\eta,M_0)$. Note that since the colors are complementary, the same properties also hold for the blue graph. Call a part $V_i$ \emph{blue} if $d_B(V_i) \geq \frac 12$ and \emph{red} otherwise.

	Suppose first that at least $m' \geq pm$ of the parts are blue and rename the parts so that $V_1,\ldots,V_{m'}$ are these blue parts. We build a reduced graph $G$ whose vertex set is $v_1,\ldots,v_{m'}$ by making $\{v_i,v_j\}$ an edge if and only if $(V_i,V_j)$ is $\eta$-regular and $d_R(V_i,V_j) \geq \delta$. Suppose that some vertex in $G$, say $v_1$, has degree at most $(1-p^{k-1}-\eta/p)m'-1$. 
	Since $v_1$ has at most $\eta m \leq \eta m'/p$ non-neighbors coming from irregular pairs $(V_1,V_j)$, this means that there are at least $p^{k-1} m'$ parts $V_j$ such that $(V_1,V_j)$ is $\eta$-regular and  $d_B(V_1,V_j) \geq 1- \delta$. Let $J$ be the set of all these indices $j$ and  $U=\bigcup_{j \in J}V_j$ be the union of all of these $V_j$. We then have
	\begin{equation}\label{eq:density-lb}
		e_B(V_1,U)=\sum_{j \in J} e_B(V_1,V_j) \geq \sum_{j \in J} (1- \delta) |V_1||V_j|=(1- \delta) |V_1||U|.
	\end{equation}
	Let $V_1' \subseteq V_1$ denote the set of vertices $v \in V_1$ with $e_B(v,U) \geq (1-2 \delta)|U|$. Then we may write
	\begin{equation}\label{eq:density-ub}
		e_B(V_1,U)=\sum_{v \in V_1'} e_B(v,U)+\sum_{v \in V_1 \setminus V_1'} e_B(v,U) \leq |V_1'||U|+(1-2 \delta)|V_1 \setminus V_1'| |U|.
	\end{equation}
	Combining inequalities (\ref{eq:density-lb}) and (\ref{eq:density-ub}), we find that $|V_1'| \geq \frac 12 |V_1|$, where every vertex in $V_1'$ has blue density
	at least $1-2 \delta$ into $U$. Moreover, since $\eta<\frac 16$, we may apply the $\eta$-regularity of $V_1$ to conclude that the internal blue density of $V_1'$ is at least $\frac 12- \eta \geq \frac 13$, while the hereditary property of regularity implies that $V_1'$ is $2\eta$-regular. Then the counting lemma, \cref{lem:countinglemma}, implies that $V_1'$ contains at least
	\[
		\frac{1}{k!}\left( d_B(V_1')^{\binom k2}-2\eta \binom k2 \right) |V_1'|^k \geq \frac{1}{k!}\left( 3^{-\binom k2}-2\eta \binom k2 \right) |V_1'|^k >0
	\]
	blue $K_k$, so that $V_1'$ contains at least one blue $K_k$. Every vertex of this blue $K_k$ has at least $(1-2 \delta)|U|$ blue neighbors in $U$, so the blue $K_k$ has at least $(1-2 k \delta)|U|$ blue extensions into $U$. Moreover, since we assumed that $|J| \geq p^{k-1} m' \geq p^{k} m$ and the partition is equitable, we find that $|U| \geq p^{k}N$. Therefore,
	\begin{align*}
		(1-2 k \delta)|U| &\geq (1-2 k \delta)p^k (p^{-k}+\varepsilon)n\\
		&= (1-2 k \delta)(1+p^k \varepsilon)n\\
		&\geq (1+p^k \varepsilon -4 k \delta)n\\
		&\geq n,
	\end{align*}
	since our choice of $\delta$ yields $4k \delta = (1-p)^{2k} \varepsilon \leq p^k \varepsilon$. Thus, we find that any blue $K_k$ inside $V_1'$ must have at least $n$ blue extensions, giving us our blue $B_n \up k$. 

	So we may assume that every vertex in $G$ has degree at least $(1-p^{k-1}-\eta/p)m'$.
	Recall from \eqref{eq:fpk} that $f(1-p,k) = p^{1-k}/(e^2k) \geq 1$ for $k \geq k_1(1-p)$. Since we assume that $k \geq k_2(p) = k_1(1-p)$, this implies that 
	\begin{equation}\label{eq:3k}
		p^{k-1} \leq \frac{1}{e^2k}\leq \frac 1{3(k-1)}.
	\end{equation}
	Additionally, by our choice of $\eta \leq (1-p)/(4k) \leq p/(4k)$, we know that
	\[
		\frac \eta p \leq \frac{1}{3(k-1)}.
	\]
	The previous two inequalities imply that
	\[
		1-p^{k-1}-\frac \eta p >1- \frac{1}{k-1},
	\]
	so that $G$ contains a $K_k$ by Tur\'an's theorem. Let $v_{i_1},\ldots,v_{i_k}$ be the vertices of this $K_k$ and let $C_j=V_{i_j}$ for $1 \leq j \leq k$. Then we claim that $C_1,\ldots,C_k$ is a $(k,\eta,\delta)$-blue-blocked configuration. The fact that each $C_i$ is $\eta$-regular follows immediately from our application of \cref{reglem} and the fact that $d_B(C_i) \geq \delta$ follows from the fact that we assumed $d_B(C_i) \geq \frac 12$. Finally, the definition of edges in $G$ implies that $(C_i,C_j)$ is $\eta$-regular with $d_R(C_i,C_j) \geq \delta$ for all $i \neq j$. Thus, our coloring contains a $(k,\eta,\delta)$-blue-blocked configuration with $\delta \leq (1-p) \varepsilon$ and $\eta \leq \delta^{4k^2}$, so \cref{lem:block-config-suffices} implies that the coloring contains either a red $B_{cn} \up k$ or a blue $B_n \up k$.

	We have now finished the proof if at least $pm$ of the parts $V_i$ are blue. Therefore, we may assume instead that at least $m'' \geq (1-p) m$ of the parts are red and again rename the parts so that these red parts are $V_1,\ldots,V_{m''}$. We construct a reduced graph $G$ on vertices $v_1,\ldots,v_{m''}$ by connecting $v_i$ to $v_j$ if $(V_i,V_j)$ is $\eta$-regular with $d_B(V_i,V_j) \geq \delta$. Suppose that some vertex in $G$, say $v_1$, has degree at most $(1-(1-p)^{k-1}-\eta/(1-p))m''-1$. As before, $v_1$ has at most $\eta m \leq \eta m''/(1-p)$ non-neighbors coming from irregular pairs. Thus, if we let $J$ denote the set of indices $j$ for which $(V_1,V_j)$ is $\eta$-regular with $d_R(V_1,V_j) \geq 1- \delta$, then we find that $|J| \geq (1-p)^{k-1}m'' \geq (1-p)^k m$. Thus, if $U=\bigcup_{j \in J}V_j$, then we see that $|U| \geq (1-p)^k N$, since the partition is equitable. Next, as above, we let $V_1' \subseteq V_1$ denote the set of vertices $v \in V_1$ with $e_R(v,U) \geq (1-2 \delta)|U|$ and find that $|V_1'| \geq \frac 12 |V_1|$. Therefore, as above, we know that $V_1'$ contains at least one red $K_k$ and this red $K_k$ has at least $(1-2 k \delta)|U|$ red extensions in $U$. Moreover,
	\begin{align}
		(1-2k \delta)|U|&\geq (1-2k \delta) (1-p)^k N\notag\\
		&=(1-2 k \delta)(1-p)^k (p^{-k}+\varepsilon)n \notag\\
		&=(1-2 k \delta)(c+(1-p)^k \varepsilon)n\label{eq:value-of-c}\\
		&\geq (c+(1-p)^k \varepsilon-4 k \delta)n\notag\\
		&\geq cn\label{eq:bounding-delta},
	\end{align}
	where in (\ref{eq:value-of-c}) we used the definition of $p$, which implies that $((1-p)/p)^k=c$, and in (\ref{eq:bounding-delta}) we used our choice of $\delta$ to see that $\delta \leq (1-p)^k \varepsilon/(4k)$. Thus, in this case, we can find a red $B_{cn}\up k$. 

	We may therefore assume that every vertex in $G$ has degree at least $(1- (1-p)^{k-1}-\eta/(1-p))m''$. 
	As before, we know that, since $k \geq k_2(p)$,
	\[
		(1-p)^{k-1} \leq p^{k-1} \leq \frac{1}{3(k-1)}
	\]
	and our choice of $\eta \leq (1-p)/(4k)$ implies that
	\[
		\frac{\eta}{1-p} \leq \frac{1}{3(k-1)}.
	\]
	Thus, by Tur\'an's theorem, $G$ must contain a $K_k$, with vertices $v_{i_1},\ldots,v_{i_k}$. If we let $C_j=V_{i_j}$, then $C_1,\ldots,C_k$ will be a $(k,\eta,\delta)$-red-blocked configuration, by the definition of edges in $G$ and the assumption that $N$ is sufficiently large in terms of $\varepsilon$. Thus, by \cref{lem:block-config-suffices}, we can again conclude that the coloring contains either a red $B_{cn}\up k$ or a blue $B_n \up k$. 
	
	To finish, we note that, as claimed, we may take $c_1(k) \leq ((1+o(1))\frac{\log k}{k})^k$. Indeed, for any $c$ and $k$, let $p(c,k) = (c^{1/k}+1)\inv$ and $y = y(c,k) = 1/\log[1/p(c,k)]$. Then, from \cref{lem:block-config-suffices,lem:offdiagxi}, we see that $k_2(p(c,k)) = 1+y(5+\log y+\log \log y)$. Thus, if $y\leq (1+o(1))k/\log k$, then $k \geq k_2(p(c,k))$. Since $y = 1/\log(1+c^{1/k})$, this condition is equivalent to $c^{1/k} \geq \exp((1+o(1))\frac{\log k}{k})-1= (1+o(1))\frac{\log k}{k}$, which yields the desired bound.
\end{proof}

\section{Quasirandomness} \label{sec:quasirandomness}

In the previous section, we showed that for a certain range of $c$ and $k$, the Ramsey number $r(B_{cn}\up k,B_n\up k)$ is, as $n \to \infty$, asymptotically equal to the lower bound coming from a $p$-random construction. In this section, we strengthen this result, showing that all colorings whose number of vertices is close to the Ramsey number must either be quasirandom or else contain substantially larger books than the Ramsey property implies. We make the following definition.
\begin{Def}
For $p \in [\frac 12,1)$ and $\gamma>0$, we say that a red/blue coloring of $E(K_N)$ contains \emph{$(p,\gamma)$-many books} if it contains
\begin{itemize}
    \item at least $\gamma N^k$ red $K_k$, each with at least $((1-p)^k +\gamma)N$ extensions to a red $K_{k+1}$, or
    \item at least $\gamma N^k$ blue $K_k$, each with at least $(p^k +\gamma)N$ extensions to a blue $K_{k+1}$.
\end{itemize}
\end{Def}

Here is the restatement of \cref{thm:quasirandom} in terms of $(p, \gamma)$-many books that we will prove.
\begin{thm:quasirandom}
	For every $p \in [\frac 12,1)$, there exists some $k_0 \in \N$ such that the following holds for every $k \geq k_0$. For every $\theta>0$, there exists some $\gamma>0$ such that if a red/blue coloring of $E(K_N)$ is not $(p,\theta)$-quasirandom, then it contains $(p,\gamma)$-many books.
\end{thm:quasirandom}

To prove \cref{thm:quasirandom}, we will need a few technical lemmas. At a high level, the proof closely follows the proof of the main quasirandomness theorem in \cite[Section 5]{CoFoWi}, as follows. First, we prove a strengthening of \cref{lem:offdiagxi}, which can be thought of as a stability version of that result; it says that if our vector $(x_1,\ldots,x_k)$ is bounded in $\ell_\infty$ away from the minimizing point $(p,\ldots,p)$, then the value of the function in \cref{lem:offdiagxi} is bounded away from its minimum of $1$. Using this, we can strengthen \cref{lem:block-config-suffices} to say that not only does a blocked configuration imply the existence of the desired monochromatic book, but in fact it implies the existence of a larger book unless every part of the blocked configuration is $\varepsilon$-regular to the entire vertex set. Therefore, assuming our coloring does not contain many blue $B\up k_{(p^k+\gamma)N}$ or red $B\up k_{((1-p)^k+\gamma)N}$, we will be able to repeatedly pull out vertex subsets that are $\varepsilon$-regular to the entire vertex set until we have almost partitioned all the vertices into such subsets. At that point, we can use the structure coming from this partition to deduce that the coloring is $(p,\theta)$-quasirandom, as desired.

We begin with the strengthening of \cref{lem:offdiagxi}.

\begin{lem}\label{lem:stable-offdiag-xi}
	For $p \in (0,1)$, let $k_1=k_1(p)$ be as in \cref{lem:offdiagxi}. Then, for every integer $k \geq k_1$ and any $\varepsilon_0>0$, there exists some $\delta_0>0$ such that if $x_1,\ldots,x_k \in [0,1]$ are numbers with $|x_j-p| \geq \varepsilon_0$ for some $j$, then
	\[
		p^{1-k} \prod_{i=1}^k x_i+\frac{(1-p)^{1-k}}{k}\sum_{i=1}^k (1-x_i)^k \geq 1+\delta_0.
	\]
\end{lem}

\begin{proof}
	Let
	\[
		F(x_1,\ldots,x_k)=p^{1-k} \prod_{i=1}^k x_i+\frac{(1-p)^{1-k}}{k}\sum_{i=1}^k (1-x_i)^k
	\]
	and $\varphi(y)=(1-e^y)^k$. The goal is to apply H\"older's defect formula, \cref{holderdefect}, using the strict convexity of the function $\varphi$. However, $\varphi$ is only strictly convex on the interval $(\log \frac 1k,0)$ and, in order to apply \cref{holderdefect}, we in fact need a positive lower bound on $\varphi''$, but no such bound exists for the whole interval $(\log \frac 1k,0)$. Because of this, we need to separately analyze the cases where all the variables are inside a large subinterval of $(\frac 1k,1)$ and when one of them is outside such a subinterval.

	First, suppose that one of the variables, say $x_1$, is in the interval $[0,\frac {1+\varepsilon_1}k]$, for some small constant $\varepsilon_1>0$. Then we have that 
	\begin{align*}
		F(x_1,\ldots,x_k) &\geq \frac{(1-p)^{1-k}}{k} \left( 1-x_1 \right) ^k\geq \frac{(1-p)^{1-k}}{k} \left( 1-\frac {1+\varepsilon_1}k \right) ^k. 
	\end{align*}
	From the proof of \cref{lem:offdiagxi}, we see that this quantity is strictly larger than $1$ for all $k \geq k_1(p)$, so, by choosing $\delta_0$ appropriately, we see that $F(x_1,\dots,x_k) \geq 1+\delta_0$ in this case. 
    We may therefore assume from now on that all the variables are at least $\frac {1+\varepsilon_1}k$.

	Next, suppose that there exist values $x_1,\ldots,x_{k-1} \in [\frac {1+\varepsilon_1}k,1]$ such that $F(x_1,\ldots,x_{k-1},1)=1$. We observe that
	\begin{align*}
		\left. \pardiff F{x_k} \right|_{x_k=1}=\left[ p^{1-k} \prod_{i=1}^{k-1}x_i-(1-p)^{1-k}(1-x_k)^{k-1} \right] _{x_k=1}=p^{1-k} \prod_{i=1}^{k-1} x_i >0.
	\end{align*}
	This implies that if we move from $x_k=1$ to $x_k=1- \varepsilon_2$ for some sufficiently small $\varepsilon_2$, the value of $F$ will decrease. Therefore, there will exist a vector $(x_1,\ldots,x_k)$ for which $F(x_1,\ldots,x_k)<1$, contradicting \cref{lem:offdiagxi} as long as $k \geq k_1(p)$. Thus, for every choice of $x_1,\ldots,x_{k-1} \in [\frac {1+\varepsilon_1}k,1]$, we have that $F(x_1,\ldots,x_{k-1},1)>1$. Since the space $[\frac {1+\varepsilon_1}k,1]^{k-1} \times \{1\}$ is compact, we in fact find that $F(x_1,\ldots,x_{k-1},1) \geq 1+\delta_1'$ for all $x_1,\ldots,x_{k-1} \in [\frac {1+\varepsilon_1}k,1]$, for some sufficiently small $\delta_1'$ depending on $p$ and $k$. Finally, by continuity of $F$, we have that $F(x_1,\ldots,x_k) \geq 1+\delta_1$ whenever $x_k \geq 1- \varepsilon_2$ for some other $\delta_1,\varepsilon_2>0$. Since $F$ is a symmetric function of its variables, the same conclusion holds if $x_i \geq 1- \varepsilon_2$ for any $i$. Thus, as long as we take the $\delta_0$ in the lemma statement to be smaller than $\delta_1$, we can assume from now on that $x_i \in [\frac {1+\varepsilon_1}k, 1- \varepsilon_2]$ for all $i$. 

    By \cref{multjensen}, subject to the constraint $\prod_{i=1}^k x_i = z$, the term $\frac 1k \sum_{i=1}^k (1-x_i)^k$ is minimized when $x_i= z^{1/k}$ for all $i$. As in the proof of \cref{lem:offdiagxi}, this shows that 
    $F(x_1,\dots,x_k) \geq \psi(z^{1/k})$,
    where $\psi(w) = p^{1-k}w^k + (1-p)^{1-k}(1-w)^k$. The function $\psi$ has a global minimum at $w=p$, where its value is $1$. This shows that $F(x_1,\dots,x_k)\geq 1+\delta_0$ if $\ab{z^{1/k}-p}\geq \varepsilon_3$, for some $\varepsilon_3>0$ depending on $p,k$, and $\delta_0$. Moreover, by picking $\delta_0$ sufficiently small, we can make $\varepsilon_3$ as small as we wish. Therefore, we may now assume that $z^{1/k} = p \pm \varepsilon_3$, which implies that $\log(z^{1/k}) = (\log p) \pm \varepsilon_4$ for some $\varepsilon_4>0$, which can also be made arbitrarily small by picking $\delta_0$ appropriately.
    
    We are now ready to apply H\"older's defect formula. First, we observe that for $y \in [\log\frac{1+\varepsilon_1}k,\log(1-\varepsilon_2)]$, we have
    \[
        \varphi''(y) = k e^y (1 - e^y)^{k - 2} (k e^y - 1)\geq k \cdot \frac{1+\varepsilon_1}k \cdot \varepsilon_2^{k-2} \cdot \varepsilon_1 \eqqcolon m,
    \]
    where $m$ is a fixed, strictly positive constant. Let $y_i = \log x_i$ for $1 \leq i \leq k$, so that $\frac 1k \sum_{i=1}^k y_i = \log(z^{1/k})$. We assumed that $\ab{x_j-p}\geq \varepsilon_0$ for some $j$, which implies that $\ab{y_j-\log p} \geq \varepsilon_0$ as well, since the derivative of $\log x$ is bounded below by $1$ on the interval $(0,1)$. Therefore, choosing $\delta_0$ small enough that $\varepsilon_4<\varepsilon_0$, we see that
    \[
        \frac 1k \sum_{i=1}^k (y_i-\log(z^{1/k}))^2 \geq \frac 1k (y_j-\log(z^{1/k}))^2 \geq \frac 1k (\varepsilon_0-\varepsilon_4)^2,
    \]
    since $\log(z^{1/k}) = (\log p)\pm \varepsilon_4$ and $\ab{y_j-\log p}\geq \varepsilon_0$. Hence, by \cref{holderdefect}, we have that
    \begin{align*}
        F(x_1,\dots,x_k) &= p^{1-k} z +\frac{(1-p)^{1-k}}k \sum_{i=1}^k (1-x_i)^k\\
        &= p^{1-k}z + (1-p)^{1-k}\cdot \frac 1k \sum_{i=1}^k \varphi(y_i) \\
        &\geq p^{1-k}z + \varphi(\log(z^{1/k}))+ \frac m{2k}(\varepsilon_0-\varepsilon_4)^2\\
        &=\psi(z^{1/k}) + \frac m{2k} (\varepsilon_0-\varepsilon_4)^2\\
        &\geq 1+\delta_0,
    \end{align*}
    where we use the fact that $\psi(w)\geq 1$ for all $w \in [0,1]$ and take $\delta_0$ sufficiently small.
\end{proof}

Using \cref{lem:stable-offdiag-xi}, we can now prove the following strengthening of \cref{lem:block-config-suffices}, which says that if we have a blocked configuration $C_1,\ldots,C_k$ and many vertices whose blue density into $C_i$ is far from $p$, then we can find a substantially larger monochromatic book than what is guaranteed by \cref{lem:block-config-suffices}.
\begin{lem}\label{lem:density-p}
	Fix $p \in [\frac 12,1)$ and let $k \geq k_2(p)$, where $k_2$ is the constant from \cref{lem:block-config-suffices}.
	Suppose $0<\varepsilon_0<\frac 14$ and let $\delta_0=\delta_0(\varepsilon_0)$ be the parameter from \cref{lem:stable-offdiag-xi}. Let $0<\delta\leq (1-p)\delta_0 \varepsilon_0$ and $0<\eta\leq \delta^{4k^2}$ and suppose that $C_1,\ldots,C_k$ is either a $(k,\eta,\delta)$-red-blocked configuration or a $(k,\eta,\delta)$-blue-blocked configuration in a red/blue coloring of $K_N$. Define
	\[
		B_i=\{v\in K_N: |d_B(v,C_i)-p| \geq \varepsilon_0\}.
	\]
	Then the following hold: 
	\begin{lemenum}
	    \item\label{lemitem:density-basic} If $|B_i| \geq \varepsilon_0 N$ for some $i$, then the coloring contains a blue $B_{(p^k+\beta)N}\up k$ or a red $B_{((1-p)^k+\beta)N}\up k$, where $\beta=(1-p)^k\delta_0 \varepsilon_0/2$.
	    \item\label{lemitem:density-counting} If, in addition, $|C_i| \geq \tau N$ for all $i$ and some $\tau>0$, then there exists some $0<\gamma<\beta$ depending on $\varepsilon_0,\tau,$ and $\delta$ such that the coloring contains $(p,\gamma)$-many books.
	\end{lemenum}
\end{lem}

\begin{proof}
	We may assume without loss of generality that $|B_1| \geq \varepsilon_0 N$. As in the proof of \cref{lem:block-config-suffices}, we need to split into two cases, depending on whether $C_1,\ldots,C_k$ is blue-blocked or red-blocked. We begin by assuming that it is $(k,\eta,\delta)$-red-blocked.

	First, as in the proof of \cref{lem:block-config-suffices}, observe that each $C_i$ contains at least one red $K_k$ and there is at least one blue $K_k$ spanning $C_1,\ldots,C_k$. Moreover, if we assume that $|C_i| \geq \tau N$ for all $i$, then \cref{lem:countinglemma} shows that the number of blue $K_k$ spanning $C_1,\ldots,C_k$ is at least
	\[
		\left( \prod_{1 \leq i <j \leq k}d_B(C_i,C_j)- \eta \binom k2 \right) \prod_{i=1}^k |C_i| \geq \left( \delta^{\binom k2}-\eta \binom k2 \right) (\tau N)^k \geq \left( \frac{\delta^{\binom k2}\tau^k}{2} \right) N^k
	\]
	and similarly, with an additional factor of $1/k!$, for the number of red $K_k$ inside each $C_i$. 

	For a vertex $v$ and $i \in [k]$, let $x_i(v)=d_B(v,C_i)$. \cref{lem:offdiagxi} implies that, for any $v \in V$,
	\[
		p \left( p^{-k}\prod_{i=1}^k x_i(v) \right) +(1-p) \left( \frac{(1-p)^{-k}}{k}\sum_{i=1}^k (1-x_i(v))^k \right) \geq 1.
	\]
	Additionally, if $v \in B_1$, then $|x_1(v)-p|\geq \varepsilon_0$, so \cref{lem:stable-offdiag-xi} implies that, for $v \in B_1$,
	\[
		p \left( p^{-k}\prod_{i=1}^k x_i(v) \right) +(1-p) \left( \frac{(1-p)^{-k}}{k}\sum_{i=1}^k (1-x_i(v))^k \right) \geq 1+\delta_0.
	\]
	Adding these two equations up over all $v \in V$ shows that 
	\[
		p \left( p^{-k} \sum_{v \in V}\prod_{i=1}^k x_i(v) \right) +(1-p) \left( \frac{(1-p)^{-k}}{k}\sum_{i=1}^k \sum_{v \in V}(1-x_i(v))^k \right) \geq N+\delta_0|B_1| \geq (1+\delta_0 \varepsilon_0)N.
	\]
	That is, a $(p,1-p)$-weighted average of two quantities is at least $(1+\delta_0 \varepsilon_0)N$, which implies that one of the quantities must itself be at least $(1+\delta_0 \varepsilon_0)N$. Suppose first that 
	\begin{equation*}
		p^{-k} \sum_{v \in V} \prod_{i=1}^k x_i(v) \geq (1+\delta_0 \varepsilon_0)N.
	\end{equation*}
	Let $Q$ be a uniformly random blue $K_k$ with one vertex in each of $C_1,\ldots,C_k$. Let $\alpha=\delta^{k^2} \leq \prod_{i<j} d_B(C_i,C_j)$, so that $\eta \leq \delta^{4k^2} = \alpha^4 \leq \alpha^3/k^2$. Therefore, applying  \cref{lem:randomclique} to each $v$ and summing up the result, we find that the expected number of blue extensions of $Q$ is at least
	\begin{align*}
		\sum_{v \in V} \left( \prod_{i=1}^k x_i(v)-4 \alpha \right) &\geq (p^k+ p^k\delta_0 \varepsilon_0-4 \alpha)N.
	\end{align*}
	Next, observe that
	\begin{equation}\label{eq:alpha-bound}
		4 \alpha=4 \delta^{k^2} \leq \frac{\delta^k }2 \leq \frac{((1-p)\delta_0 \varepsilon_0)^k}{2} \leq \frac{(1-p)^k \delta_0 \varepsilon_0}{2} \leq \frac{p^k \delta_0 \varepsilon_0}{2},
	\end{equation}
	which implies that the expected number of blue extensions of $Q$ is at least $(p^k+\beta)N$, where $\beta=(1-p)^k \delta_0 \varepsilon_0/2$. Thus, there exists a blue $B_{(p^k+\beta)N}\up k$, proving (a) in this case. Moreover, if we assume that $|C_i| \geq \tau N$ for all $i$, then our earlier computation shows that $Q$ is chosen uniformly at random from a set of at least $\kappa N^k$ monochromatic cliques, where $\kappa=\delta^{\binom k2} \tau^k/2$. We may therefore apply \cref{lem:markov-consequence} with $\xi=p^k+\beta$ and $\nu=p^k+\gamma$, for some appropriately chosen $0<\gamma<\beta$, to conclude that in this case our coloring contains at least $\gamma N^k$ blue cliques, each with at least $(p^k+\gamma)N$ blue extensions, proving (b).

	Therefore, we may assume that the second term in the weighted average is the large one, i.e., that
	\[
		\frac{(1-p)^{-k}}{k}\sum_{i=1}^k \sum_{v \in V}(1-x_i(v))^k\geq (1+\delta_0 \varepsilon_0)N,
	\]
	which implies that, for some $i$,
	\[
		\sum_{v \in V} (1-x_i(v))^k \geq (1-p)^k(1+ \delta_0 \varepsilon_0)N.
	\]
	Therefore, if $Q$ is now a random red $K_k$ inside this $C_i$,  \cref{lem:randomclique} implies that the expected number of red extensions of $Q$ is at least
	\[
		\sum_{v \in V} \left[ (1-x_i(v))^k-4 \alpha \right] \geq \left[ (1-p)^k+(1-p)^k \delta_0 \varepsilon_0-4 \alpha \right] N.
	\]
	But, by  (\ref{eq:alpha-bound}), $4 \alpha \leq (1-p)^k \delta_0 \varepsilon_0/2$, which implies that the expected number of red extensions of $Q$ is at least $((1-p)^k+\beta)N$, proving (a). As before, if we also assume that $|C_i| \geq \tau N$ for all $i$, then we may apply \cref{lem:markov-consequence} with $\kappa=\delta^{\binom k2}\tau^k/2k!$, $\xi=(1-p)^k+\beta$, and $\nu = (1-p)^k+\gamma$ to find that our coloring contains at least $\gamma N^k$ red $K_k$, each with at least $((1-p)^k+\gamma)N$ red extensions for some appropriately chosen $\gamma \in (0,\beta)$, yielding (b). This concludes the proof of the lemma in the case where $C_1,\ldots,C_k$ is a $(k,\eta,\delta)$-red-blocked configuration.

	As in the proof of \cref{lem:block-config-suffices}, the other case, where $C_1,\ldots,C_k$ is a $(k,\eta,\delta)$-blue-blocked configuration, follows in an almost identical fashion. 
	We define $y_i(v)=d_R(v,C_i)$ for all $v \in V$ and $i \in [k]$ and let $q=1-p$. We then apply \cref{lem:offdiagxi,lem:stable-offdiag-xi} with these $y$ variables and with $q$ instead of $p$. The remaining details are  exactly the same. 
\end{proof}

Next, we strengthen \cref{lem:density-p} by showing that not only does every part of a blocked configuration have density roughly $p$ to most vertices, but it is in fact $(p,\varepsilon)$-regular to the entire vertex set. Here, by saying that a pair of vertex subsets $(X,Y)$ is \emph{$(p,\varepsilon)$-regular}, we mean that $|d(X',Y')-p| \leq \varepsilon$ for every $X' \subseteq X$, $Y' \subseteq Y$ with $|X'| \geq \varepsilon |X|$, $|Y'| \geq \varepsilon |Y|$. Note that $(p,\varepsilon)$-regularity is equivalent, up to a linear change in the parameters, to $\varepsilon$-regularity with density $p \pm \varepsilon$.

\begin{lem}\label{lem:block-config-regular}
	Fix $p \in [\frac 12, 1)$ and let $k \geq k_2(p)$.
	Suppose $0<\varepsilon_1<\frac 14$,  $\varepsilon_0=\varepsilon_1^2/2$, and let $\delta_0=\delta_0(\varepsilon_0)$ be the parameter from \cref{lem:stable-offdiag-xi}. Let $0<\delta\leq (1-p)\delta_0 \varepsilon_0$ and $0<\eta\leq \varepsilon_1 2^{-4k^2}\delta^{4k^2}$ and suppose that $C_1,\ldots,C_k$ is either a $(k,\eta,\delta)$-red-blocked or a $(k,\eta,\delta)$-blue-blocked configuration in a red/blue coloring of $K_N$. 
	Then the following hold:
	\begin{lemenum}
	    \item \label{lemitem:regular-basic} If, for some $i$, the pair $(C_i,V)$ is not $(p,\varepsilon_1)$-regular in blue, then the coloring contains a blue $B\up k_{(p^{k}+\beta)N}$ or a red $B\up k _{((1-p)^k+\beta)N}$, where $\beta=(1-p)^k\delta_0 \varepsilon_0/2$.
	    \item\label{lemitem:regular-counting} If, in addition, $|C_i| \geq \tau N$ for all $i$ and some $\tau>0$, then the coloring contains $(p,\gamma)$-many books for some $0<\gamma<\beta$ depending on $\varepsilon_1,\delta$, and $\tau$.
	\end{lemenum}
\end{lem} 	

\begin{proof}
	Without loss of generality, suppose that $(C_1,V)$ is not $(p,\varepsilon_1)$-regular in blue. Then there exist $C_1' \subseteq C_1, D \subseteq V$ with $|C_1'| \geq \varepsilon_1 |C_1|, |D| \geq \varepsilon_1 N$ such that $|d_B(C_1',D)-p| >\varepsilon_1$. Assume first that $d_B(C_1',D) \geq p+\varepsilon_1$. Let $D_1 \subseteq D$ denote the set of vertices $v \in D$ with $d_B(v,C_1')< p +\frac{\varepsilon_1}{2}$ and let $D_2=D \setminus D_1$. Then we have that
	\[
		\left(p+\varepsilon_1\right) |C_1'||D| \leq \sum_{v \in D_1} e_B(v,C_1')+\sum_{v \in D_2} e_B(v,C_1') \leq \left( p+ \frac {\varepsilon_1}2 \right) |C_1'||D|+|C_1'||D_2|,
	\]
	which implies that $|D_2| \geq \frac{\varepsilon_1}{2}|D| \geq \frac{\varepsilon_1^2}{2}N=\varepsilon_0 N$,  where each $v \in D_2$ has $d_B(v,C_1') \geq p+\frac{\varepsilon_1}{2}$. Now, consider the $k$-tuple of sets $C_1',C_2,\ldots,C_k$; by the hereditary property of regularity, we see that this is a $(k,\eta',\delta')$-blocked configuration, where $\eta'=\eta/\varepsilon_1$ and $\delta'=\delta- \eta \geq \delta/2$. This implies that $\delta' \leq (1-p) \delta_0 \varepsilon_0$ and $\eta' \leq (\delta')^{4k^2}$. 
	Therefore, we may apply \cref{lemitem:density-basic} to the $(k,\eta',\delta')$-blocked configuration $C_1',C_2,\ldots,C_k$ to conclude that the coloring contains a blue $B_{(p^k+\beta)N}\up k$ or a red $B_{((1-p)^k+\beta)N}\up k$. Moreover, if we assume that $|C_i| \geq \tau N$ for all $i$, then $|C_i'| \geq \varepsilon_1 \tau N$ for all $i$, where $C_i'=C_i$ if $i\geq 2$. Thus, \cref{lemitem:density-counting} implies that in this case the coloring contains $(p,\gamma)$-many books for some $0<\gamma<\beta$ depending on $\varepsilon_1,\delta$, and $\tau$. 

	To complete the proof of the lemma, we also need to check the case where $d_B(C_1',D) \leq p- \varepsilon_1$. However, the proof is essentially identical: we find a subset $D_2 \subseteq D$ such that every vertex $v \in D_2$ has $d_B(v,C_1') \leq p- \frac{\varepsilon_1}{2}$ and such that $|D_2| \geq \frac{\varepsilon_1}{2}|D|$ and then the rest of the proof is as above. 
\end{proof}

Our next technical lemma gives the inductive step for our proof of \cref{thm:quasirandom}. The proof mimics that of \cref{thm:offdiagupper}, except that the vertex set is split into parts that were already pulled out as regular and a part that has not yet been touched. Inside the untouched part, we build a reduced graph and use it to find either many large monochromatic books or a blocked configuration, at which point \cref{lem:block-config-regular} implies that the induction can continue.

\begin{lem}\label{lem:pull-out-regular}
	Fix $p \in [\frac 12,1)$ and let $k \geq k_2(p)$. Fix $0<\varepsilon\leq p/(20 k)$ and suppose that the edges of the complete graph $K_N$ with vertex set $V$ have been red/blue colored. 
	Suppose that $A_1,\ldots,A_\ell$ are disjoint subsets of $V$ such that $(A_i,V)$ is $(p, \varepsilon^2)$-regular for all $i$. Let $W= V \setminus (A_1 \cup \dotsb \cup A_\ell)$ and suppose that $|W| \geq \varepsilon N$. Then either there is some $A_{\ell+1} \subseteq W$ such that $(A_{\ell+1},V)$ is $(p, \varepsilon^2)$-regular or else the coloring contains $(p,\gamma)$-many books for some $\gamma>0$ depending on $\varepsilon$, $p,$ and $k$.
\end{lem}

\begin{proof}
	Let $\varepsilon_1=\varepsilon^2$, $\varepsilon_0=\varepsilon_1^2/2,$ and $\delta_0=\delta_0(\varepsilon_0)$ be the parameter from \cref{lem:stable-offdiag-xi} and set $\delta=(1-p)\delta_0 \varepsilon_0$, $\eta=\varepsilon^2 2^{-4k^2} \delta^{4k^2}$, $\beta=k p^{k-1}\varepsilon^2$, and $\beta'=4 \varepsilon$. We apply \cref{reglem} to the subgraph induced on $W$, with parameters $\eta$ and $M_0=1/\eta$, to obtain an equitable partition $W=W_1 \sqcup \dotsb \sqcup W_m$, where $M_0 \leq m \leq M=M(\eta,M_0)$. Call a part $W_i$ \emph{blue} if $d_B(W_i) \geq \frac 12$ and \emph{red} otherwise. As in the proof of \cref{thm:offdiagupper}, we first assume that at least $m' \geq pm$ of the parts are blue and rename them so that $W_1,\ldots,W_{m'}$ are the blue parts. 

	We build a reduced graph $G$ on vertex set $w_1,\ldots,w_{m'}$, connecting $w_{i_1}$ and $w_{i_2}$ by an edge if $(W_{i_1},W_{i_2})$ is $\eta$-regular and $d_R(W_{i_1},W_{i_2}) \geq \delta$. Suppose that $w_1$ has at most $(1-p^{k-1}-\beta'/p-\eta/p)m'-1$ neighbors in $G$. Since $w_1$ has at most $\eta m \leq \eta m'/p$ non-neighbors coming from irregular pairs, 
	this means that there are at least $(p^{k-1}+\beta'/p)m'$ parts $W_j$ with $2 \leq j \leq m'$ such that $(W_1,W_j)$ is $\eta$-regular and $d_B(W_1,W_j) \geq 1-\delta$. Let $J$ be the set of these indices $j$ and set $U=\bigcup_{j \in J} W_j$.
	By the counting lemma, \cref{lem:countinglemma}, $W_1$ contains at least $\frac{1}{k!}\left(2^{-\binom k2}-\eta \binom k2\right)|W_1|^k$ blue copies of $K_k$ and
	\[
		\frac{1}{k!}\left(2^{-\binom k2}- \eta \binom k2\right)|W_1|^k \geq \frac{2^{-k^2}}{k!} \left( \frac{|W|}{M} \right) ^k \geq \left(\frac{ \varepsilon N}{k 2^k M}\right)^k,
	\]
	where we use that $\eta\leq \delta^{4k^2} \leq \delta^{\binom k2}/\binom k2$ and that $2^{-\binom k2}-\delta^{\binom k2} >2^{-k^2}$, along with our assumption that $|W| \geq \varepsilon N$. If we set $\kappa=( \varepsilon/k 2^k M)^k$, then this implies that $W_1$ contains at least $\kappa N^k$ blue $K_k$. If we pick a uniformly random such blue $K_k$, then \cref{lem:randomclique} with $\alpha=\delta^{k^2} \leq 2^{-\binom k2}\leq d_B(W_1)^{\binom k2}$ implies that its expected number of blue extensions inside $U$ is at least 
	\begin{align*}
		\sum_{u \in U} \left(d_B(u,W_1)^k-4 \alpha\right) \geq \left[ (1- \delta)^k- \delta^k \right] |U| \geq (1-2k \delta) |U|,
	\end{align*}
	where we first use Jensen's inequality applied to the convex function $x \mapsto x^k$ to lower bound $\sum_u d_B(u,W_1)^k$ by $(1- \delta)^k|U|$ and then use that $(1- \delta)^k \geq 1-k \delta$ and $4 \delta^{k^2} \leq \delta^k \leq k \delta$. Since we assumed that $J$ was large and the partition is equitable, we find that
	\[
		|U| \geq (p^{k-1}+\beta'/p)m' |W_j|\geq (p^{k}+\beta')|W|.
	\]
	Thus, a random blue $K_k$ inside $W_1$ has at least $(1-2k \delta)(p^{k}+\beta')|W|$ blue extensions in $W$. 

	Now, suppose that instead of just $w_1$ having low degree in $G$, we have a set of at least $\varepsilon m$ vertices $w_j \in V(G)$, each with at most $(1-p^{k-1}-\beta'/p-\eta/p)m'-1$ neighbors in $G$. Let $S$ be the set of these $j$ and $T=\bigcup_{j \in S} W_j$. By the above argument, for every $j\in S$, we have that $W_j$ contains at least $\kappa N^k$ blue $K_k$ such that a uniformly average one among them has at least $(1-2k \delta)(p^{k}+\beta')|W|$ blue extensions into $W$. Moreover, we have that 
	\[
		|T|=|S||W_j| \geq \varepsilon m \frac{|W|}{m}=\varepsilon |W| \geq \varepsilon^2 |V|.
	\]
	We may therefore apply the $(p,\varepsilon^2)$ regularity of $(A_i, V)$ to conclude that $d_B(A_i,T)=p \pm \varepsilon^2$ for all $i$. Thus, if we pick $j \in S$ randomly, then $\E[d_B(W_j,A_i)]=p \pm \varepsilon^2$. Therefore, if we first sample $j \in S$ randomly and then pick a random blue $K_k$ inside $W_j$, then \cref{lem:randomclique} implies that this random blue $K_k$ will have in expectation at least
	\begin{align*}
		\sum_{a \in A_i} \left( d_B(a,W_j)^k-4 \delta ^{k^2} \right) &\geq \left[ \left( p- \varepsilon^2 \right) ^k- \delta^k \right] |A_i| \\
		&\geq \left[ p^k \left( 1- \frac{k\varepsilon^2}{p} \right) -\delta^k \right] |A_i|\\
		&\geq p^k \left( 1- \frac{2k\varepsilon^2}{p} \right) |A_i|
	\end{align*}
	blue extensions into $A_i$, again by Jensen's inequality. This implies that this random $K_k$ has in expectation at least $(1 -2k \varepsilon^2/p) p^{k}|A_1 \cup \dotsb \cup A_\ell|$ extensions into $A_1 \cup \dotsb \cup A_\ell$. Adding up the extensions into this set and into $W$, its complement, shows that this random blue $K_k$ has in expectation at least $\xi N$ blue extensions, where $\xi$ is a weighted average of $(1-2k \varepsilon^2/p)p^{k}$ and $(1-2k \delta)(p^{k}+\beta')$, and where the latter quantity receives weight at least $\varepsilon$, since $|W| \geq \varepsilon N$. Thus,
	\begin{align*}
		\xi &\geq (1- \varepsilon)\left(1-\frac{2k \varepsilon^2}p\right) p^{k}+\varepsilon(1-2k \delta)(p^{k}+\beta')\\
		&\geq \left(1-\frac{2k \varepsilon^2}p- \varepsilon\right)p^k+\varepsilon(1-2k \delta)(1+p^{-k} \beta')p^k\\
		&\geq \left(1-\frac{2k \varepsilon^2}p- \varepsilon\right)p^k+\varepsilon \left(1+\frac{3k \varepsilon}p\right)p^k\\
		&=p^k\left(1+\frac{k \varepsilon^2}p\right)\\
		&=p^k+\beta,
	\end{align*}
	where we used the definition of $\beta$, the fact that $2k\delta<p^{-k}\beta'/4$, that $(1-x/4)(1+x) \geq 1+x/2$ for all $x \in [0,1]$, and that $p^{-k} \beta' \geq 6k \varepsilon/p$, which follows since $\beta'=4\varepsilon$ and, as in the proof of \cref{lem:offdiagxi},  $p^{1-k} \geq e^2 k \geq \frac 32 k$ for $k \geq k_2(p)$. Therefore, by \cref{lem:markov-consequence}, we can find at least $\gamma N^k$ blue $K_k$, each with at least $(p^{k}+\gamma)N$ blue extensions, for some $\gamma<\beta$ depending on $\varepsilon$ and $\beta$ and, thus, only on $\varepsilon$, $p$, and $k$. 

	Therefore, we may assume that in $G$, all but $\varepsilon m\leq \varepsilon m'/p$ of the vertices have degree at least $(1-p^{k-1}-\beta'/p- \eta/p)m'$. 
	Hence, the average degree in $G$ is at least
	\begin{align*}
		\left(1- \frac{\varepsilon}p\right) \left( 1-p^{k-1} -\frac{\beta'}p- \frac \eta p \right)m' & \geq \left( 1-p^{k-1} - \frac{6 \varepsilon }p\right)m' \geq \left(1- p^{k-1}-\frac 1{3k} \right)m',
	\end{align*}
	since $\beta'=4 \varepsilon$, $\eta \leq \varepsilon$, and $\varepsilon \leq p/(20k)$. By \eqref{eq:3k}, the fact that $k \geq k_2(p)$ implies that $p^{k-1} \leq 1/(3k)$. Therefore, the average degree in $G$ is greater than $(1-1/(k-1))m'$, so, by Tur\'an's theorem, $G$ will contain a $K_k$. Let $w_{i_1},\ldots,w_{i_k}$ be the vertices of this $K_k$ and let $C_j=W_{i_j}$ for $1 \leq j \leq k$. Then, by the definition of $G$, we see that $C_1,\ldots,C_k$ is a $(k,\eta,\delta)$-blue-blocked configuration with $|C_i| \geq \tau N$ for all $i$, where $\tau=\varepsilon/M$ depends only on $\varepsilon$, $p$, and $k$.  Thus, by \cref{lemitem:regular-counting}, we see that either the coloring contains $(p,\gamma)$-many books for some $\gamma$ depending on $\varepsilon,p,$ and $k$ or else $(C_j,V)$ is $(p,\varepsilon^2)$-regular for all $j$. In the latter case, we can set $A_{\ell+1}=C_1$ (or any other $C_j$) and get the desired result. 

	Now, we need to assume instead that at least $m'' \geq (1-p)m$ of the parts $W_i$ are red. However, just as in the proof of \cref{thm:offdiagupper}, the argument is essentially identical: we first rule out the existence of too many low-degree vertices in the reduced graph by counting extensions to $W$ and to $A_1\cup \dotsb \cup A_\ell$ and then apply Tur\'an's theorem to find a $K_k$ in the reduced graph, which completes the proof by 
	\cref{lemitem:regular-counting}.
\end{proof}

By repeatedly applying \cref{lem:pull-out-regular} until $W$ has fewer than $\varepsilon N$ vertices, we can partition $K_N$ into a collection of subsets $A_i$ such that $(A_i,V)$ is $(p,\varepsilon^2)$-regular, plus a small remainder set $A_{\ell+1}$ about which we have no such information. Our final technical lemma shows that such a structural decomposition suffices to conclude that the coloring is $(p,\theta)$-quasirandom.
\begin{lem}\label{lem:partition-implies-quasirandomness}
	Let $\varepsilon\leq \theta/3$. Suppose we have a partition
	\[
		V(K_N)=A_1 \sqcup \dotsb \sqcup A_\ell \sqcup A_{\ell+1}
	\]
	where $(A_i,V)$ is $(p,\varepsilon)$-regular for each $1 \leq i\leq \ell$ and $|A_{\ell+1}| \leq \varepsilon N$. Then the coloring is $(p,\theta)$-quasirandom.
\end{lem}

\begin{proof}
	Fix disjoint $X,Y \subseteq V(K_N)$. We need to check that
	\[
		\left|e_B(X,Y)- p |X||Y| \right| \leq \theta N^2.
	\]
	First, observe that if $|Y| \leq \varepsilon N$, then 
	\[
		\left|e_B(X,Y)- p |X||Y| \right| \leq |X| |Y| \leq \varepsilon N^2 \leq \theta N^2.
	\]
	Therefore, from now on, we may assume that $|Y| \geq \varepsilon N$. For $1 \leq i \leq \ell+1$, let $X_i=A_i \cap X$ and define $I_X=\{1 \leq i \leq \ell:|X_i| \geq \varepsilon |A_i|\}$. Then we have that 
	\[
		\sum_{i \notin I_X} |X _i| \leq |A_{\ell+1}|+ \varepsilon \sum_{i=1}^\ell |A_i| \leq 2\varepsilon N.
	\]
	We now write
	\[
		e_B(X,Y)-p |X||Y|=\sum_{i=1}^{\ell+1} \left( e_B(X_i,Y)-p |X_i||Y| \right) .
	\]
	We will split this sum into two parts, depending on whether $i \in I_X$ or not. First, suppose that $i \in I_X$. Then $|X_i| \geq \varepsilon |A_i|$ and $|Y| \geq \varepsilon |V|$, 
	so we may apply the $(p, \varepsilon)$-regularity of $(A_i,V)$ to conclude that
	\[
		\sum_{i \in I_X}\left|e_B(X_i,Y)-p|X_i||Y|\right|=\sum_{i \in I_X}\left|d_B(X_i,Y)-p\right| |X_i||Y| \leq \sum_{i \in I_X}\varepsilon |X_i||Y| \leq \varepsilon |X||Y| \leq \varepsilon N^2.
	\]
	On the other hand, since $\sum_{i \notin I_X} |X_i| \leq 2 \varepsilon N$, we have that
	\[
		\sum_{i \notin I_X} \left| e_B(X_i,Y)- p |X_i||Y| \right| \leq  |Y|\sum_{i \notin I_X} |X_i| \leq |Y| (2 \varepsilon N) \leq 2\varepsilon N^2.
	\]
	Adding these together, we conclude that
	\[
		\left| e_B(X,Y) -p |X||Y|\right| \leq 3 \varepsilon N^2 \leq \theta N^2,
	\]
	as desired. 
\end{proof}

With all these pieces in place, the proof of \cref{thm:quasirandom} becomes quite straightforward.

\begin{proof}[Proof of \cref{thm:quasirandom}]
	Fix $p \in [\frac 12,1)$ and suppose $k \geq k_0 \coloneqq k_2(p)$. Fix $\theta>0$ and set $\varepsilon=\min \{\theta/3,p/(20k)\}$. Let $\gamma=\gamma(\theta,p,k)$ be the parameter from \cref{lem:pull-out-regular}. Suppose we are given a coloring of $K_N$ without $(p,\gamma)$-many books. We wish to prove that the coloring is $(p,\theta)$-quasirandom. We inductively apply \cref{lem:pull-out-regular} to find a sequence $A_1,\ldots,A_\ell$ of vertex subsets such that $(A_i,V)$ is $(p,\varepsilon^2)$-regular for all $i$ and, therefore, $(p,\varepsilon)$-regular for all $i$. We continue until the remainder set $W=V \setminus(A_1 \cup \dotsb \cup A_\ell)$ satisfies $|W| \leq \varepsilon N$, at which point the assumptions of \cref{lem:pull-out-regular} are no longer met, so we set $A_{\ell+1} = W$. However, at this point, we can apply \cref{lem:partition-implies-quasirandomness} to conclude that our coloring is indeed $(p,\theta)$-quasirandom.
\end{proof}

\subsection{The converse}

In this section, we prove a converse to \cref{thm:quasirandom}, which implies that not containing $(p,\gamma)$-many books is an equivalent characterization of $p$-quasirandomness.

\begin{thm}\label{thm:quasirandom-converse}
    Fix $k \geq 2$ and $p \in (0,1)$. Then, for every $\gamma>0$, there exists some $\theta>0$ such that the following holds for every $(p,\theta)$-quasirandom coloring of $E(K_N)$ with $N$ sufficiently large. Apart from fewer than $\gamma N^k$ exceptions, every red $K_k$ has $((1-p)^k \pm \gamma)N$ extensions to a red $K_{k+1}$ and every blue $K_k$ has $(p^k \pm \gamma)N$ extensions to a blue $K_{k+1}$. In particular, the coloring 
does not contain $(p,\gamma)$-many books.
\end{thm}

\begin{rem}
	In this direction, there is no dependence between $p$ and the range of $k$ for which the result holds. As we know from the fact that the $k$-partite structure is the extremal structure for small $c$, such a dependence is necessary in the forward direction. 
	However, here, all we are saying is that almost all monochromatic books in a quasirandom coloring are of essentially the correct size, that is, asymptotic to what they would be in a random coloring.
\end{rem}

\begin{proof}
	We will use the well-known result of Chung, Graham, and Wilson \cite{ChGrWi}, that a quasirandom coloring contains roughly the correct count of any fixed monochromatic subgraph. Specifically, for every $\delta>0$, there is some $\theta>0$, such that, in any $(p,\theta)$-quasirandom coloring of $E(K_N)$, 
	\begin{align*}
		B(K_k)&\coloneqq\#(\text{blue }K_k)=p^{\binom k2} \binom Nk \pm \delta N^k,\\
		B(K_{k+1})&\coloneqq\#(\text{blue }K_{k+1})=p^{\binom{k+1}2} \binom N{k+1} \pm \delta N^{k+1},\\
		B(K_{k+2}-e)&\coloneqq\#(\text{blue }K_{k+2}-e)=p^{\binom{k+2}2-1}\binom N{k+2} \binom{k+2}2 \pm \delta N^{k+2},
	\end{align*}
	where $K_{k+2}-e$ is the graph formed by deleting one edge from $K_{k+2}$; note that for this count we have an extra factor of $\binom {k+2}2$ to account for the fact that this graph is not vertex-transitive. On the other hand, we can observe that every blue copy of $K_{k+2}-e$ corresponds to two distinct extensions of a single blue $K_k$ to a blue $K_{k+1}$. Therefore,
	\[
		B(K_{k+2}-e)=\sum_{Q} \binom{\#(\text{blue extensions of }Q)}2,
	\]
	where the sum is over all blue $K_k$. Let $\ext_B(Q)$ denote the number of blue extensions of $Q$. Then we can also observe that $\sum_Q \ext_B(Q)$ counts the total number of ways of extending a blue $K_k$ into a blue $K_{k+1}$, which is precisely $(k+1)B(K_{k+1})$, since each blue $K_{k+1}$ contributes exactly $k+1$ terms to this sum.

	Now, we consider the quantity 
	\[
		E=\sum_{Q\text{ a blue }K_k} (\ext_B(Q)-p^{k}N)^2.
	\]
	On the one hand, we have that if $\delta \geq 1/N$, then
	\begin{align*}
		&E=\sum_Q \ext_B(Q)^2-2p^kN \sum_Q \ext_B(Q)+\sum_Q p^{2k} N^2\\
		&=\left( 2 \sum_Q \binom{\ext_B(Q)}2+\sum_Q \ext_B(Q) \right) -2p^k N(k+1)B(K_{k+1})+p^{2k} N^2 B(K_k)\\
		&=2 B(K_{k+2}-e)+(1-2p^k N)(k+1)B(K_{k+1})+p^{2k} N^2 B(K_k)\\
		&\leq 2p^{\binom{k+2}2-1} \binom N{k+2}\binom{k+2}2-2p^k N(k+1)p^{\binom{k+1}2}\binom N{k+1}+p^{2k} N^2 p^{\binom k2} \binom Nk +5k \delta N^{k+2}\\
		&=p^{\frac{k^2+3k}2} \binom Nk (-N+k^2+k) + 5k \delta N^{k+2}\\
		&<5k \delta N^{k+2}.
	\end{align*}
	On the other hand, suppose there were at least $\gamma N^k/2$ blue $K_k$ with at least $(p^{k}+\gamma)N$ or at most $(p^k-\gamma)N$ blue extensions. Then, by only keeping these cliques in the sum defining $E$, we would have that
	\begin{align*}
		E&=\sum_Q (\ext_B(Q)-p^{k}N)^2 \geq \frac{\gamma N^k}2 (\gamma N)^2=\frac{\gamma^3}2 N^{k+2}.
	\end{align*}
	Therefore, if we pick $\delta< \gamma^3/10k$, we get a contradiction. The same argument with $p$ replaced by $1-p$ and blue replaced by red shows that there are also fewer than $\gamma N^k/2$ red $K_k$ with at least $((1-p)^k+\gamma)N$ or at most $((1-p)^k-\gamma)N$ red extensions. This proves the theorem, since the total number of exceptional cliques is at most $\gamma N^k$.
\end{proof}

\section{Concluding remarks}
Putting together the main results of this paper, we obtain the following picture. For every $k \geq 2$, there exist two numbers $c_0(k),c_1(k) \in (0,1]$ such that if $0 <c\leq c_0$, then $r(B_{cn}\up k,B_n \up k) =k(n+k-1)+1$, while if $1\geq c\geq c_1$, then $r(B_{cn} \up k, B_n \up k) = (c^{1/k}+1)^k n+o_k(n)$. 
Moreover, in both these  regimes, there are stability results: there exist $c_0'(k)\leq c_0(k)$ and $c_1'(k)\geq c_1(k)$ such that for $0 < c \leq c_0'$, all the near-extremal colorings are close to $k$-partite,\footnote{For concreteness, we can fix $c'_0(k)$ as coming from an application of Theorem~\ref{thm:goodness-stability} with $\theta = 1/k^3$.}
while for all $1 \geq c \geq c_1'$, all near-extremal colorings are quasirandom. Of course, the most natural question remaining is to understand what happens in the interval $(c_0',c_1')$, where our results say nothing. Note that this gap is real, since below $c_0'$ all extremal colorings must be $k$-partite, whereas above $c_1'$ all extremal colorings must be quasirandom. On the other hand, it is possible that there is no gap between $c_0$ and $c_1$, since it is conceivable that at the point where the random and $k$-partite constructions yield comparable lower bounds on $r(B_{cn}\up k, B_n\up k)$, both are tight.

This question about the gap really comprises at least two separate questions: what happens for fixed $k$ and what happens as $k \to \infty$? 
To address the second question first, our results give some indication. Indeed, we have shown that both $c_0(k)$ and $c_1(k)$ tend to $0$ as $k \to \infty$ and thus the gap interval shrinks as $k \to \infty$. More precisely, we have that 
$$c_0(k) \leq \left((1+o(1))\frac{\log k}k\right)^k \leq c_1(k) \leq \left((1+o(1))\frac{\log k}k\right)^k.$$ 
Moreover, the results of \cite{FoHeWi} show that $1/c_0$ is at most single-exponential in a power of $k$. 
On the other hand, because we used the regularity lemma, our upper bound for $1/c_0'$ is only of tower-type. However, it seems likely that the methods of \cite{FoHeWi} could also be adapted to improve this.

The other question is what happens for fixed $k$. Here, our understanding is much more limited, even for the simplest case $k=2$. In this case, Nikiforov and Rousseau \cite{NiRo} proved that $c_0(2)=1/6$, in the sense that, for all $c<1/6$ and all $n$ sufficiently large, $r(B_{cn}\up 2, B_n \up 2) = 2n+3$, whereas, for any $c>1/6$ and all $n$ sufficiently large, there is a construction showing that $r(B_{cn} \up 2, B_n \up 2)>2n+3$. Curiously, our results do not say anything non-trivial about $c_1(2)$, other than the fact that the random bound is correct for $c=1$; in other words, we cannot prove that $c_1(2) < 1$ and in fact believe this to not be the case.

\begin{conj}
    For every $c<1$, the random bound for $r(B_{cn}\up 2,B_n \up 2)$ is not tight. In other words, there exists some $\beta=\beta(c)>0$ such that $r(B_{cn}\up 2,B_n \up 2) \geq ((\sqrt c+1)^2 + \beta)n$ for all $n$ sufficiently large.
\end{conj}

Of course, this conjecture is really only the tip of an iceberg, with the general open question being to understand $r(B_{cn}\up 2,B_n \up 2)$ for $c\in (1/6,1)$ and $n\to \infty$. There are many conjectures one could make about the behavior of this quantity as a function of $c$; for instance, perhaps there are a number of thresholds in the interval $(1/6,1)$ at which new extremal structures emerge, each dictating the value of $r(B_{cn}\up 2, B_n \up 2)$ until the next threshold. Because we know that the random bound is correct for $c=1$ and that quasirandom colorings are the only extremal ones, such a sequence of extremal examples would need to converge, in some appropriate sense, to the quasirandom coloring as $c \to 1$. However, at the moment we are not even able to conjecture a single such extremal structure or threshold.

\paragraph{Acknowledgments.} We are grateful to the anonymous referee for helpful comments which improved the presentation of this paper.


\end{document}